\documentclass[11pt]{amsart}
\usepackage{amsmath, amsthm, amssymb}

\usepackage[utf8]{inputenc}
\usepackage{fullpage}
\usepackage[all,cmtip]{xy}
\usepackage{color}
\usepackage{microtype}
\usepackage{fourier}
\usepackage{enumitem}

\newtheorem{definition}[subsubsection]{Definition}

\newcommand{\hb}{\mathbb{H}}
\newcommand{\cb}{\mathbb{C}}

\newcommand{\pb}{\mathbb{P}}
\newcommand{\Q}{\mathbb{Q}}

\newcommand{\ac}{A}
\newcommand{\zc}{Z}
\newcommand{\oc}{\mathcal{O}}

\newcommand{\rc}{R}

\newcommand{\ic}{\mathcal{I}}
\newcommand{\kc}{\mathcal{K}}
\newcommand{\xc}{{X}}
\newcommand{\yc}{Y}
\newcommand{\tc}{T}
\newcommand{\qc}{Q}

\newcommand{\hc}{\mathcal{H}}
\newcommand{\nc}{\mathcal{N}}
\newcommand{\jac}{\mathcal{J}\!\!ac}
\newcommand{\vc}{V}

\newcommand{\eb}{\mathbf{e}}\newcommand{\fb}{\mathbf{f}}

\newcommand{\Qb}{\mathbf{Q}}

\newcommand{\Td}{\operatorname{Td}}

\newcommand{\XT}{{\widetilde{X}}}
\newcommand{\xt}{\widetilde{x}}
\newcommand{\nut}{\widetilde{\nu}}
\newcommand{\mut}{\widetilde{\mu}}

\providecommand{\ev}{\operatorname{ev}}

\providecommand{\Todd}{\operatorname{Td}}
\providecommand{\Ch}{\operatorname{ch}}

\providecommand{\hess}{\mathop{dd^c}\nolimits}

\providecommand{\Gr}{\operatorname{Gr}}
\providecommand{\limi}{{\operatorname{lim}}}
\providecommand{\Spec}{\operatorname{Spec}}

\def\n{{\mid\!\mid}}
\def\and{\textrm{ and }}
\def\bcov{{\scriptscriptstyle{\textrm{BCOV}}}}

\def\quillen{{\scriptscriptstyle{\textrm{Q}}}}

\theoremstyle{plain}
\newtheorem{theo}{Theorem}[section]
\newtheorem{lem}[theo]{Lemma}

\newtheorem{prop}[theo]{Proposition}
\newtheorem{cor}[theo]{Corollary}

\theoremstyle{remark}
\newtheorem{remark}[theo]{Remark}

\newcommand\op[1]{\operatorname{#1}}
\newcommand\bb[1]{\mathbb{#1}}
\newcommand\calo{\mathcal O}

\newcounter{tmp}

\setenumerate[0]{label=(\alph*)}

\begin{document}
 \title{Singularities of metrics on Hodge bundles and their topological invariants}
\author{Dennis Eriksson}
\author{Gerard Freixas i Montplet}
\author{Christophe Mourougane}
\address{Dennis Eriksson \\ Department of Mathematics \\ Chalmers University of Technology and Gothenburg University }
\email{dener@chalmers.se}

\address{Gerard Freixas i Montplet \\ C.N.R.S. -- Institut de Math\'ematiques de Jussieu - Paris Rive Gauche}
\email{gerard.freixas@imj-prg.fr}

\address{Christophe Mourougane\\Institut de Recherche Math\'ematique de Rennes (IRMAR)}
\email{christophe.mourougane@univ-rennes1.fr}

\begin{abstract}
We consider degenerations of complex projective Calabi--Yau varieties and study the singularities of $L^2$, Quillen and BCOV metrics on Hodge and determinant bundles. The dominant and subdominant terms in the expansions of the metrics close to non-smooth fibers are shown to be related to well-known topological invariants of singularities, such as limit Hodge structures, vanishing cycles and log-canonical thresholds. We also describe corresponding invariants for more general degenerating families in the case of the Quillen metric.
\end{abstract}
\maketitle
\setcounter{tocdepth}{1}
\tableofcontents

\section{Introduction}
\begingroup
\setcounter{tmp}{\value{theo}}
\setcounter{theo}{0} 
\renewcommand\thetheo{\Alph{theo}}
In this article we study the singularities of several natural metrics on combinations of Hodge type bundles, for degenerating families of complex projective algebraic varieties. In particular we provide topological interpretations of invariants associated to logarithmic singularities of these metrics. Our original motivation was a metrical approach to the canonical bundle formula for families of Calabi--Yau varieties \cite{fujino-mori}. The first instance of this formula goes back to Kodaira \cite[Thm. 12]{kodaira:canonical}, and describes the relative canonical bundle of an elliptic surface in terms of a positive modular part and some topological invariants of the singular fibers. We were thus naturally led to the study of Hodge type bundles, their metrics and behavior close to singular fibers. 

 As a matter of motivation, a classical example to keep in mind is the Hodge bundle $f_\ast \omega_{X/S}$ for a family of compact Riemann surfaces $f \colon X\to S$, endowed with its canonical $L^{2}$-metric or a Quillen metric on its determinant bundle (cf. section \ref{quillenbackground}). The latter topic is the main focus of the work of Bismut--Bost \cite{bismutbost}. In the semi-stable case, they describe the singularities and the curvature current of the Quillen metric on the determinant of the Hodge bundle. In the special case where $S$ is the unit disk and there is a unique singular fiber $X_0$ at $0 \in S$, the principal part of the curvature current is of the form $$\frac{\# \mathrm{sing}(X_0)}{12}\delta_0,$$ where $\delta_0$ is the Dirac current at 0 and $\# \mathrm{sing}(X_0)$ is the number of singular points in the fiber $X_0$.

In this article, we study analogues of this phenomenon for $L^2$-metrics on Hodge bundles for Calabi--Yau families (Theorem \ref{A}), Quillen metrics on determinant bundles (Theorem \ref{B}) and the so-called BCOV metric, which has found applications in mirror symmetry for Calabi--Yau 3-folds (Theorem \ref{C}). 

To state our contributions, for the purpose of this introduction, we suppose that $f: X \to S$ is a flat, projective map of complex manifolds of relative dimension $n$, and $S$ is the unit disc with parameter $s$. We suppose the fibers $X_s = f^{-1}(s)$ connected and smooth for $s\neq 0$ (we say that $f$ is generically smooth). We also assume that $X$ carries a fixed K\"ahler metric. We denote by  $K_{X/S} = K_X \otimes K_S^{-1}$ the relative canonical bundle. 

\begin{theo} \label{A} Suppose the general fiber of $X \to S$ is Calabi--Yau, i.e. with trivial canonical bundle. 
Let $\eta$ be a local holomorphic frame of the line bundle $f_\ast K_{X/S}$. Then if we define 
$$\n \eta \n^2_{s} = \left|\int_{X_s} \eta \wedge \overline{\eta} \ \right|$$
we have
\begin{displaymath}
-\log\n\eta\n^2=\alpha\log|s|^2-\beta\log|\log|s|^2|+O(1)
\end{displaymath}
where 
\begin{enumerate}
\item $\alpha = 1 - c_{X_{0}}(f) \in [0,1) \cap \bb Q$. Here $c_{X_{0}}(f)$ is the log-canonical threshold of $(X,-B,X_{0})$ along $X_{0}$, where $B$ is the divisor of the evaluation map $f^{\ast}f_{\ast}(K_{X/S})\to K_{X/S}$. Moreover,  $\exp(-2 \pi i \alpha)$ is the eigenvalue of the semi-simple part of the monodromy acting on the graded piece $\Gr_F^n H^n_\limi $  of the middle limit Hodge structure of $X \to S$. 

\item $\beta=\delta(X,X_{0}) \in [0,n]\cap \bb N$ is the degeneracy index of $(X,X_{0})$, computed through the geometry of the special fiber and $K_{X/S}$. Moreover, $\beta + n$ is the mixed Hodge structure weight of the 1-dimensional space $\Gr_F^n H^n_\limi $.
\end{enumerate}
\end{theo}
This statement summarizes the results in section \ref{section:L2}. On the smooth locus, the curvature of the $L^2$-metric is the K\"ahler form of the modular Weil-Petersson metric. Hence Theorem \ref{A} indicates the necessary correction of the Hodge bundle so that the $L^2$-metric becomes good in the sense of Mumford. Another example of application of the theorem is for morphisms with isolated ordinary quadratic singularities. We show that $\alpha=\beta=0$, and hence the $L^{2}$ metric is continuous in this case.

Versions of Theorem \ref{A} already appeared in the work of other authors, in slightly different forms. For instance, the degeneracy index and log-canonical threshold have also been studied by Halle--Nicaise \cite[Thm. 6.2.2]{NicaiseHalle}. In the context of $\ell$-adic cohomology, they establish the analogous relationship as in the theorem above. There is also related work of Berman \cite[Sec. 3]{Berman} on the asymptotics of $L^2$-metrics in terms of log-canonical thresholds. More recently Boucksom--Jonsson \cite{BoucksomJonsson} study asymptotics of volume forms in relationship with non-archimedean limits. Actually, the argument we provide for the asymptotics in terms of $c_{X_{0}}(f)$ and $\delta(X,X_{0})$ is a specialization of the computations in \emph{loc. cit.}, and was communicated to us by S. Boucksom, whom we warmly thank.


In sections~\ref{section:Quillen} and~\ref{section:BCOV}, we shift our interest to the determinant line bundle endowed with a Quillen type metric,
instead of the direct image of the relative canonical bundle endowed with the $L^2$ metric.
The main feature is that, after normalizing the metric,
this bundle still detects the variation in moduli in its smooth part,
and has a degeneration mainly governed by the singular fibers, and weakly depending on their germs of embedding.
Suppose now that $\vc$ is a hermitian vector bundle on $X$ and let $\lambda(V)_Q$ be the determinant of the cohomology of $V$, equipped with the Quillen metric. This has a singularity at 0, and our aim is to provide a topological measure of it. If $\sigma$ is a local holomorphic frame of $\lambda(V)$, then Yoshikawa \cite{yoshikawa} proves that
$$\log \n \sigma\n^2_\quillen = \left(\int_{X_0} \mathcal{Y}(X/S,V)\right)\log |s|^2 +R(s)\quad\text{as}\quad s\to 0,$$
where $\mathcal{Y}(X/S,V)$ is a certain cohomology class and $R(s)$ is a continuous function of $s$. In this article we study and generalize this class for families of varieties over a general parameter space $S$, and whose total space $X$ is not necessarily smooth. This uses and underlines the Nash blowup instead of the Gauss morphism. The latter was actually introduced by Bismut \cite{bismut} and then further exploited by Yoshikawa \cite{yoshikawa}. Our approach allows us to study the class $\mathcal{Y}(X/S,V)$ from the point of view of Fulton's intersection theory, which exhibits functoriality properties to the effect that we can use moduli space arguments in computations. For the formulation of the theorem,  for simplicity, let $\vc$ be the trivial line bundle and set $Y(X/S) = \int_{X_0} \mathcal{Y}(X/S,V)$.
\begin{theo} \label{B}

\begin{enumerate}
    \item Suppose that $X$ is not necessarily smooth and is a family of hypersurfaces in  $\pb^n$ parametrized by $S$. Then $$Y(X/S) = \frac{(-1)^{n+1}}{(n+2)!} \int_{X_0} c_{n+1}^{X_0}(\Omega_{X/S})$$ where $c_{n+1}^{X_0}(\Omega_{X/S})$ denotes the localized top Chern class of $X \to S$. 
    \item The same formula holds if $X \to S$ is a family of $K3$ or abelian surfaces, with $X$ smooth and $K_X$ trivial. Then one can conclude that $$Y(X/S) = \frac{-1}{24}(\chi(X_\infty)-\chi(X_0)).$$
   \end{enumerate}

\end{theo}
 In fact, for general families $X \to S$ with $X$ smooth, we have the fundamental relation $$\int_{X_0}c_{n+1}^{X_0}(\Omega_{X/S}) = (-1)^n(\chi(\xc_\infty)-\chi(\xc_0)).$$ The expression $\chi(\xc_\infty)-\chi(\xc_0)$ is the total dimension of the vanishing cycles of the family, i.e. the difference between the topological Euler characteristics of the special fiber $X_0$ and a general fiber $X_\infty$. 
 
 The developments abutting to Theorem \ref{B} are the object of section \ref{section:Quillen}. We stress here that the intersection theoretic approach is well suited to other geometric settings. For instance, in the ``arithmetic situation" (i.e. $S$ is the spectrum of a discrete valuation ring of mixed characteristic), the Yoshikawa class can still be defined and may be seen as a discriminant, meaning a measure of bad reduction. An example of this principle was studied by the first author in \cite{dennis}, and applied in the study of Quillen metrics on degenerating Riemann surfaces \cite{ErikssonQuillen}. This was a source of inspiration for the present work.

We now turn our attention to a particular combination of Hodge type bundles. For a smooth family $f\colon X\to S$ one can consider the vector bundles $R^q f_\ast \Omega^p_{X/S}$ coming from the Hodge filtration on relative de Rham cohomology. Taking weighted determinants of these vector bundles, one introduces the BCOV line bundle (named after Bershadsky--Cecotti--Ooguri--Vafa)
$$\lambda_\bcov = \bigotimes_{p=0}^n \lambda(\Omega^p_{X/S})^{(-1)^p p}= \bigotimes_{p, q=0}^n \det \left(R^q f_\ast \Omega^p_{X/S}\right)^{(-1)^{p+q} p}.$$
Following Fang--Lu--Yoshikawa \cite{fang-lu-yoshikawa}, after a suitable rescaling of the Quillen metric on $\lambda_{\bcov}$, one defines the BCOV metric.  For a family of Calabi--Yau varieties, this is independent of the initially chosen K\"ahler metric, and its curvature is given by the modular Weil--Petersson form. Therefore it is an intrinsic invariant of the family. As \emph{loc. cit.} illustrates, for applications to mirror symmetry in physics, it is important to determine the singularities of the BCOV metric under degeneration. Hence, let us now assume that $f\colon X\to S$ is only generically smooth. The line bundle $\lambda_{\bcov}$ (initially defined on the smooth locus) has a natural extension to $S$, called the K\"ahler extension, which we denote by $\widetilde{\lambda}_\bcov$. Then, the BCOV metric on $\lambda_{\bcov}$ can be seen as a singular metric on $\widetilde{\lambda}_{\bcov}$. The last statement of this introduction summarizes our results on the singularities of the BCOV metric on $\widetilde{\lambda}_{\bcov}$, as discussed in section \ref{section:BCOV}.

\begin{theo} \label{C} Suppose that $K_X$ is trivial. Let $\eta$ be a local holomorphic frame of 
$\ \widetilde{\lambda}_\bcov $. Then,
\begin{enumerate}
\item\label{C:a} the asymptotic expansion of the BCOV metric is
$$-\log\n\eta\n_\bcov^2= \alpha_{\bcov} \log |s|^2 - \frac{\chi(X_\infty)}{12}\beta \log|\log |s|^2| + O(1).$$
Here 
 $$\alpha_{\bcov}=\frac{9n^{2}+11n+2}{24}(\chi(\xc_{\infty})-\chi(\xc_{0}))
    +\frac{\alpha}{12}\chi(X_{\infty})$$
    and $\alpha, \beta$ are as in Theorem \ref{A}. In particular, $\alpha_\bcov$ is expressed in terms of vanishing cycles and the topological Euler characteristic of a general fiber. 
    \item if the monodromy action on $H^{n}_{\lim}$ is unipotent (\emph{e.g.} if $f$ is semi-stable), then $\alpha_{\bcov}$ further simplifies to
    \begin{displaymath}
        \alpha_{\bcov}=\frac{9n^{2}+11n+2}{24}(\chi(\xc_{\infty})-\chi(\xc_{0})).
    \end{displaymath}
    \item\label{C:c} if $f$ has only isolated ordinary quadratic singularities and $n \geq 2$, then
    \begin{displaymath}
        \alpha_{\bcov}=\frac{9n^{2}+11n+2}{24}\#\mathrm{sing}(X_{0}),
    \end{displaymath}
    so that
     \begin{displaymath}
        -\log\n\eta\n_\bcov^2= \frac{9n^{2}+11n+2}{24}\# \mathrm{sing}(X_{0}) \log |s|^2 + O(1).
    \end{displaymath}
\end{enumerate}
\end{theo}
Such families with trivial canonical bundle are commonly known as \emph{Kulikov families}, named after work of Kulikov on semi-stable degenerations of K3 surfaces \cite{kulikov}. Examples in other dimensions are known to exist \cite{KawamataNamikawa, Lee}. Another situation of Kulikov family is when $f$ has relative dimension $n\geq 2$ and presents only isolated singularities so that the Kulikov assumption in the third point of Theorem \ref{C} is automatic. In fact, we provide a general closed formula for the logarithmic divergence, without any assumption on $K_{X}$. In any event, Theorem \ref{C} describes the necessary correction to the BCOV metric on $\widetilde{\lambda}_{\bcov}$ in order to obtain a good hermitian metric. 

The expression of $\alpha_{\bcov}$ in \ref{C:c} for ordinary quadratic singularities was first observed by Yoshikawa (private communication with the authors). Our approach is based on independent ideas, relying on the general expression \ref{C:a} and the fact that $\alpha=\beta=0$ for this type of singularities.

\textbf{Acknowledgements}: The authors would like to thank Bo Berndtsson,
S\'ebastien Boucksom, Lars Halvard Halle, Johannes Nicaise and Ken-Ichi Yoshikawa for interesting discussions and remarks. 
We also would like to extend our gratitude to the Hausdorff Institute for Mathematics in Bonn, where this project originated, for their support and hospitality, as well as the UMI of the CNRS in Montreal.

\endgroup
\setcounter{theo}{\thetmp}

\section{Degeneration of $L^{2}$-metrics on the Hodge bundle}\label{section:L2}
\subsection{Background on the Hodge bundle for Calabi--Yau and Kulikov families}\label{generalities} Let $f~:~\xc\to S$ be a proper flat morphism with connected fibers of dimension $n$,
from a complex manifold $\xc$ to a smooth complex curve $S$.
We will refer to such a map as \emph{a family}. Assume that $f$ is generically smooth (or submersive) with respect to the Zariski topology. The relative cotangent sheaf $\Omega_{\xc/S}$ (or sheaf of relative K\"ahler differentials) then fits into a short exact sequence
\begin{eqnarray}\label{sigmastar}
0\to f^\ast \Omega_S\stackrel{df^t}\to \Omega_\xc\to \Omega_{\xc/S}\to 0.
\end{eqnarray}
The exactness on the left is guaranteed by the generic smoothness assumption. The relative canonical bundle is defined to be $$K_{\xc/S}\colon=K_\xc\otimes f^\ast K_S^{-1}.$$
It coincides with $\Lambda^{n}\Omega_{\xc/S}$ at the points where $f$ is submersive. 

Assume now that the smooth fibers of $f$ have trivial canonical bundle. Then the direct image sheaf $f_\ast(K_{\xc/S})$, the Hodge bundle, is locally free of rank 1. Indeed, it is a torsion-free sheaf on a smooth curve, and hence locally free. Moreover, on a Zariski dense open subset of $S$ it has rank 1. The evaluation map of line bundles 
\begin{equation}\label{eq:def-ev-map}
	\ev~:~f^\ast f_\ast(K_{\xc/S})\to K_{\xc/S}
\end{equation}
is an isomorphism over smooth fibers, by base change, and injective. We denote by $B$ its zero divisor. By construction, $B$ is a divisor supported in the the singular fibers of $f$ and depends on the model $X$. The injectivity of the evaluation map implies that $B$ is effective. By construction it is supported on singular fibers. With this notations, we have the relation
\begin{align}\label{eq:eval}K_{\xc/S}=f^\ast f_\ast(K_{\xc/S})\otimes\oc_\xc(B).\end{align}
 We observe that $B$ cannot contain any full fiber of $\pi$.
For this, let $s$ be a local parameter on the curve $S$ centred at a point $s_0$.
Let $\eta$ be a local section of $\pi_\ast(K_{\xc})$ on an open set $U$ of $S$,
not divisible by $s$ in $f_\ast(K_{\xc})$.
If the zero divisor $B$ of $\ev$ contained the whole fiber $X_{s_0}$,
then $\frac{\ev f^\ast\eta}{s}$ would be a form in $K_\xc(\pi^{-1}(U))$.
But it is not of the form $\ev f^\ast\mu$ for some $\mu$ in $f_\ast(K_{\xc})(U)$, hence contradicting the surjectivity of $\ev$ on the open tube $(\pi^{-1}U)$. 

In the case when $B$ is trivial, we call $f\colon X\to S$ a \emph{Kulikov family}. Kulikov models are in general difficult to describe. For families of K3 surfaces Kulikov \cite{kulikov} established the existence of such models in the semi-stable case. In arbitrary dimension, examples are obtained by smoothing of suitable normal crossings varieties \cite{KawamataNamikawa, Lee}. Finally, we remark that if the the special fiber has at least two components, their intersection is part of the $f$-singular locus $Z$ and any component therein is of dimension at least $n+1-2$. Hence if $Z$ is of dimension at most $(n+1)-3 = n-2$, then any singular fiber $X_0$ is necessarily irreducible and reduced. We infer that $B= \emptyset$ in this case. A particular instance of this fact is given by morphisms of relative dimension $n\geq 2$ and isolated singularities.

\subsection{Log-canonical threshold, degeneracy index and the singularities of the $L^2$-metric} \label{lct}
Let $f\colon\xc\to S$ be a generically smooth family between complex manifolds, of relative dimension $n$, and where $S$ is a curve. Assume that the smooth fibers of $f$ have trivial canonical bundle, so that $f_{\ast}(K_{\xc/S})$ is a line bundle. On the smooth locus in $S$, this line bundle affords an intrinsic $L^{2}$ or Hodge metric. If $\eta$ is a nonvanishing $n$-form on a smooth fiber $X_{s}$, then 
\begin{displaymath}
	\|\eta\|_{L^{2},s}^2=\frac{1}{(2\pi)^{n}}\left |\int_{X_{s}}\eta\wedge\overline{\eta}\right |.
\end{displaymath}
If $\eta$ extends to a trivialization of $f_{\ast}(K_{\xc/S})$ in a neighborhood of $s$, then $\|\eta\|_{L^{2},s}$ changes smoothly with $s$. The question is to analyze the behavior of the $L^{2}$-metric close to the singular locus of $f$ in $S$. For the sake of simplicity, in the sequel we assume that $S$ is a disk centered at $0$ and there is at most a singular fiber at $0$.

The formation of $f_{\ast}(K_{X/S})$ is invariant under blowups along regular centers in the special fiber. 
Therefore, for the purpose of analyzing the $L^2$-metric, we may assume (after a Hironaka resolution) that the singular fiber of $f\colon\xc\to S$ decomposes into irreducible components $a_j E_j$, with $E_j$ smooth, meeting with normal crossings $$X_0=\sum a_jE_j.$$
We write the zero divisor of the evaluation map \eqref{eq:eval}, that is a relative canonical divisor, in the form $$B=\sum (b_j-1)E_j.$$
Following Koll\'ar \cite[Sec. 8, esp. Def. 8.1]{kollar:singpairs} (see also Berman \cite[Sec. 3.4, esp. Prop. 3.8]{Berman}), we define the \emph{log-canonical threshold of $(\xc, -B, X_{0})$ along $X_{0}$}  by $$c_{X_{0}}(\xc,-B,X_0)=\min_{j} \left(\frac{b_j}{a_j}\right).$$ As in \emph{loc. cit.}, we will allow the abuse of notation $c_{X_{0}}(f)$ for $c_{X_{0}}(X,-B,X_{0})$. In addition, we define 

$$b(\xc,X_0):=\max\left\lbrace\sharp{J}\mid \cap_J E_j\neq\emptyset \and \forall j\in J, \frac{b_j}{a_j}=c_{X_{0}}(\xc,-B,X_0)\right\rbrace.$$
Notice that $b(\xc, X_{0})-1$ is the degeneracy index $\delta(X,X_{0})$ defined by Halle--Nicaise \cite[Def. 6.2.1]{NicaiseHalle}. 

The log-canonical threshold and the degeneracy index govern the asymptotics of the $L^{2}$-metric close to the singular locus.
\begin{prop}\label{prop:lct}
Let $\eta$ be a holomorphic frame of $f_{\ast}K_{\xc/S}$. Then
 the $L^2$-metric on $f_\ast K_{\xc/S}$ degenerates close to $0$ as
$$-\log\n\eta\n^2_{L^2}=(1-c_{X_{0}}(f))\log |s|^2-(b(\xc,X_0)-1)\log|\log |s|^2|+O(1),$$
where $s$ is the local coordinate on $S$. \end{prop}
\begin{proof}
The isomorphism $K_{\xc/S}=\det\Omega_{\xc/S}=\det\frac{\Omega_\xc}{f^\ast\Omega_S}$ on the smooth part of $f$ yields a description of the map
\begin{eqnarray*}
 f^\ast K_S\otimes K_{\xc/S}&\to& K_\xc\\
 f^\ast ds\otimes [u]&\mapsto&f^\ast ds\wedge u.
\end{eqnarray*}
The effective divisor $B$ relates $K_{\xc/S}$ and $f^\ast f_\ast K_{\xc/S}$ by
\begin{eqnarray}\label{otimes-wedge}
f^\ast K_S\otimes f^\ast f_\ast K_{\xc/S}=\oc_\xc(-B)\otimes K_\xc.
\end{eqnarray}
 Choose a point $x_0\in X_0$. Denote by $J(x_0):=\{j, x_0\in E_j\}$.
 Choose local coordinates $(z_1,z_2,\cdots,z_{n+1})$ on $\xc$ centred at $x_0$
 such that for $j\in J(x_0)$, $E_j$ is given by $z_j=0$ and the maps $f$ becomes, locally around $x_0$, 
 $$f : (z_1,z_2,\cdots,z_{n+1})\mapsto s=\prod_{j\in J(x_0)}z_j^{a_j}.$$
 The isomorphism~(\ref{otimes-wedge}) shows the existence of an open covering $(U_\alpha)$ of $\xc$ by coordinate charts
 and invertible holomorphic functions $f_\alpha$ such that on $U_\alpha$,
 $$f^\ast ds\wedge\ev(\eta)=f_\alpha \prod_{j\in J(x_0)} z_j^{b_j-1}dz_1\wedge dz_2\wedge\cdots\wedge dz_{n+1}.$$
 We choose a partition of the unity $(\phi_\alpha)$ built from an open covering of $\xc$ where the previous simplifications hold.
 Choose a $j_0\in J(x_0)$ such that $\frac{b_{j_0}}{a_{j_0}}=\min_{J(x_0)}\frac{b_j}{a_j}$
 and note that $$dz_1\wedge dz_2\wedge\cdots\wedge dz_{n+1}
 =(-1)^{j_0}\frac{z_{j_0}}{a_{j_0}}f^\ast\frac{ds}{s}\wedge dz_1\wedge dz_2\wedge\cdots dz_{j_0-1}\wedge dz_{j_0+1}\wedge\cdots\wedge dz_{n+1}.$$
 We introduce the changes of variables $z_j=e^\frac{\rho_j}{a_j}e^{i\theta_j}$  for $j\in J(x_0)$ and $z_k=r_ke^{i\theta_k}$ for $k\in\{1,\cdots,n+1\}-J(x_0)$.
The set of integration is defined by $|z_i|\leq 1$ and $\prod_{j\in J(x_0)}z_j^{a_j}=s$,
in other words by $\rho_j\leq 0$, $0\leq r_k\leq 1$ and $\sum_{j\in J(x_0)} a_j\theta_j=\arg(s)$ and
$\sum_{j\in J(x_0)} \rho_j=\log|s|$.
 \begin{eqnarray*}
 \lefteqn{ \frac{1}{(2\pi)^{n}|s|^2} \int_{X_s}\phi_\alpha\frac{|f_\alpha(z)|^2}{|a_{j_0}|^2}\prod_{j\in J(x_0)} |z_j|^{2b_j-2}|z_{j_0}|^2
  |dz_1|^2 |dz_2|^2\cdots |dz_{j_0-1}|^2 |dz_{j_0+1}|^2\cdots |dz_{n+1}|^2}\\
  &=&\frac{1}{(2\pi)^{n}|s|^2} \int_{X_s}\phi_\alpha\frac{|f_\alpha(z)|^2}{\prod_{j\in J(x_0)}|a_j|^2}\prod_{j\in J(x_0)} e^{\frac{2b_j}{a_j}\rho_j}
  \prod_{j\in J(x_0)-j_0}d\theta_j d\rho_j\prod_{k\in \{1,\cdots,n+1\}-J(x_0)} r_kdr_kd\theta_k\\
 &\equiv&C|s|^{2(\frac{b_{j_0}}{a_{j_0}}-1)}\int_{\sum \rho_j=\log|s|}
 \phi_\alpha|f_\alpha(z)|^2\prod_{j\in J(x_0)}
 e^{\left(\frac{2b_j}{a_j}-\frac{2b_{j_0}}{a_{j_0}}\right)\rho_j}
 d\rho_1d\rho_2\cdots d\rho_{j_0-1}d\rho_{j_0+1}\cdots d\rho_{n+1}.
 \end{eqnarray*}
Adding those estimates for the different $\alpha$ and neglecting bounded terms, we easily derive the desired estimate
\begin{eqnarray*}
-\log\n \eta_s\n^2_{L^2(X_s)}
 &\equiv&{(1-\min_j\left(\frac{b_{j}}{a_{j}}\right))}\log |s|^2-{\sharp\left\{j\neq j_0/ x_0\in E_j \and  \frac{b_j}{a_j}=\frac{b_{j_0}}{a_{j_0}}\right\}}\log|\log|s|^2|+O(1).
 \end{eqnarray*} 
\end{proof}
\begin{remark} \label{remarkalphasemistable} Consider now the particular case when $X \to S$ is semi-stable. Then all the $a_i=1$ and since the divisor $B$ does not contain a whole fiber (see section \ref{generalities}), there is at least one $b_i=1$. We conclude that $c_{X_{0}}(f) =\min_{j} (\frac{b_j}{a_j})=1$. If the family is moreover a Kulikov model, then all the $b_i=1$. In this case, it follows that the degeneracy index is simply the maximal number of intersecting components in the special fiber. 
\end{remark}

\subsection{The $L^2$-metric and semi-stable reduction}


Let us now examine the change of the $L^{2}$-metric under semi-stable reduction. We consider a semi-stable reduction diagram
$$\xymatrix{\yc\ar[r]^F\ar[d]_g&\xc\ar[d]^f\\T\ar[r]^\rho&S}$$
where $g$ is a semi-stable family, $\rho$ is the finite morphism $t \mapsto s= t^e$ and $F$ a generically finite morphism. 
From~\cite[lemmas 3.3 and 4.2]{MT},
we know that $(g_\ast(K_{\yc/T}), L^2)$ isometrically embeds
into $(\rho^\ast f_\ast(K_{\xc/S}),\rho^\ast L^2)$.
A local frame $\xi$ for $g_\ast(K_{\yc/T})$ hence relates to a local frame $\eta$
for $f_\ast(K_{\xc/S})$ through
$$\xi=t^a\rho^\ast\eta$$
where $a$ can be recovered by the formula
$$a=\dim_{\bb C}\frac{\rho^\ast f_\ast(K_{\xc/S})}{g_\ast(K_{\yc/T})}.$$ From the previous proposition we get
\begin{prop}\label{alphabeta}
The asymptotic of the $L^2$-metric on $f_\ast(K_{\xc/S})$ is of the shape
\begin{displaymath}
    -\log\n\eta\n^2=\alpha\log|s|^2-\beta\log|\log|s|^2|+C+O(\frac{1}{\log|s|})
\end{displaymath}
where $$\alpha=\frac{a}{e}=\frac{1}{e}\dim_{\bb C}\frac{\rho^\ast f_\ast(K_{\xc/S})}{g_\ast(K_{\yc/T})}$$ and $$\beta=b(X,X_{0})-1 = b(Y, Y_0)-1.$$
\end{prop}
\begin{remark} \label{yoshikawaloglog} 
The fact that the metric has the above shape, with $\alpha =0$ in the semi-stable case, is already stated in \cite[Thm. 6.8]{yoshikawa3}.
\end{remark}

\subsection{The $L^2$-metric via variation of Hodge structures}
Let $f~:~\xc\to \Delta$ be a proper K\"ahler morphism with connected fibres of dimension $n$
from a complex manifold $\xc$ to the complex unit disc $\Delta$, which is a holomorphic submersion on $\Delta^\times$. 
We suppose that the special fiber is a normal crossings divisor, and that the equation for $f$ is locally given by $s= z_1^{n_1} \ldots z_k^{n_k}$, where $s$ is the standard parameter on $\Delta$. Denote by $f^\times := f^{-1}(\Delta^\times) \to \Delta^\times$ the smooth part of $f$.
 Let $\gamma$ be the monodromy operator of the local system $\rc^n f^\times_\ast\cb$, 
and $\gamma = \gamma_u \gamma_s = \gamma_s \gamma_u$ be its Jordan decomposition where $\gamma_u$ is unipotent and $\gamma_s$ semi-simple.

The aim of this section is to prove the following statement. 
\begin{prop}
With the previous notations, suppose furthermore that $h^{n,0}=1$. Then,
\begin{enumerate}
    \item $\exp(-2\pi i\alpha) = \exp(2 \pi i c_{X_{0}}(f))$ is the eigenvalue of $\gamma_{s}$ acting on $\Gr_{F}^{n}H^{n}_{\limi} = F^n H^{n}_{\limi}$.
    \item $n+\beta$ is the weight of the 1-dimensional space $\Gr_{F}^{n} H^n_{\limi}$.
\end{enumerate}
\end{prop}
\begin{remark} 
The result seems to be known, and is announced in \cite[Thm. 6.2.2 (2)]{NicaiseHalle} and the authors inform us the methods amount 
to the usage of Steenbrink's constructions of the logarithmic relative de Rham complex. Our method of proof is based on a (nowadays standard) combination of Deligne extensions of local systems and Schmid's construction of the limit mixed Hodge structure. 
\end{remark}

\begin{proof}

We will use the correspondence between an element $\Qb$ in $H^n(X_\infty,\cb)$ 
and the corresponding multi-valued flat section $Q$ of the local system $\rc^n f^\times_\ast\cb$.
Let $\rho = \exp(2\pi i - ) : \hb \to \Delta^\times$ be the universal covering of the punctured unit disc, 
and for $\tau \in \hb$ set $s = \exp(2\pi i \tau)$.
Set $\Gamma = N + S$, where $N = \frac{1}{2\pi i}\log \gamma_u$ and $S = \frac{1}{2\pi i}\log \gamma_s$, 
where for $S$, we have fixed the branch of the logarithm having imaginary part in  $[0, 2\pi )$. Hence $S$ has eigenvalues in $[0,1)$.

Let $\fb_1, \ldots, \fb_N$ is a basis of $H^n(X_\infty,\cb)$. The corresponding multi-valued flat basis satisfies $f_i(\tau +1) = \gamma f_i(\tau)$. If we
define $$e_i := s^{-\Gamma} f_i = \exp(-2 \pi i \tau \Gamma) f_i$$
then we have $e_i(\tau +1) = e_i(\tau)$. 
The \emph{Deligne canonical extension}\footnote{Also called the upper extension due to the choice of the logarithm} $\hc^n$ 
of  $\rc^n f^\times_\ast\cb\otimes\oc_{\Delta^\times}$
is defined to be the locally free $\oc_\Delta$ module generated by the $e_i$'s. The Gauss--Manin connection on $R^{n}f_{\ast}^{\times}\bb C \otimes\calo_{\Delta^{\times}}$ extends to a regular singular connection on $\hc^{n}$. Its residue is readily computed in the basis $e_{i}$, and seen to coincide with $\Gamma$.

We denote by $H^{n}_{\limi}$ the limit (mixed) Hodge structure on $H^n(X_\infty,\cb)$, the cohomology of a general fiber. 
By construction,  $H^{n}_{\limi}$ is equipped with a decreasing filtration $F^p H^n(X_\infty,\cb)$, \emph{the Hodge filtration}, 
and an increasing filtration $W_k H^n(X_\infty,\cb)$, \emph{the weight filtration} built from the nilpotent operator $N$.  Moreover, $H^{n}_{\lim}$ may be identified with the fiber of $\hc^{n}$ at 0, with monodromy action given in terms of the residue of the Gauss--Manin connection $\exp(2\pi i\Gamma)$.

Let now $\Qb\in F^{n}H^{n}(X_{\infty},\bb C)$ be non-zero. From its corresponding multi-valued flat section $Q$, we construct the section of $\hc^{n}$ determined by
\begin{displaymath}
    Q_{\infty}(\tau)=\exp(-2\pi i\tau\Gamma) Q(\tau).
\end{displaymath}
This section is called the twisted period. Its fiber at $0\in\Delta$ is denoted by $\Qb_{\infty}$, and is seen as an element in $H^{n}_{\lim}$. Let $\ell$ the integer such that $\Qb_\infty$ belongs to $W_\ell$ but not to $W_{\ell-1}$.
By construction of the weight filtration, the nilpotent operator $N$ maps $W_\ell$ to $W_{\ell-2}$.
The semi-simple part $\gamma_s$ (and hence $S$) acts on $H^{n}_{\limi}$  as a mixed Hodge structure operator~\cite[Theorem 2.13]{Steenbrink-mixedonvanishing}.
Write $\omega_j:=\exp(2\pi i\lambda_j)$ where $\lambda_j$ is a non-increasing sequence of rational numbers in $[0,1)$,
for the sequence of eigenvalues of $\gamma_s$ acting on $W_\ell/W_{\ell-1}$.
Choose a basis $(\eb_j)$ of $H^n_{\lim}$ adapted to the filtration $W$.
Hence, $\Qb_\infty$ can be decomposed as
$$\Qb_\infty=\Qb^+ +\Qb'$$
where $\Qb^+:=\sum_j q_j \eb_j$, $S\eb_j=\lambda_j\eb_j+\eb'_j$ and $\Qb'$ and the $\eb'_j$ belong to $W_{\ell-1}$.
As $\gamma_s$ respects the Hodge filtration on $H^n(X_\infty,\cb)$, and as $h^{n,0}=1$, $F^nH^n(X_\infty,\cb)$
is an eigenspace for $S$, with eigenvalue, say $\lambda$. From the freeness of $(\eb_j)$, it follows that for each $j$, either $q_j=0$ or $\lambda_j=\lambda$, so that $$S(\Qb^+)=\lambda (\Qb^+)+\Qb''$$ where $\Qb''\in W_{\ell-1}$.

By the nilpotent orbit theorem~\cite{schmid}, and as shown by Kawamata \cite[Lemma 1]{kawa82},
$$f_\ast(K_{\xc/\Delta})=\iota_\ast f^\times_\ast (K_{\xc^\times/\Delta^\times})\cap \hc^n$$
where $\iota : \Delta^\times\to \Delta$ is the inclusion.
We can hence write a local frame $\eta$ for $f_\ast(K_{\xc/\Delta})$ as 
$$\eta = \sum_i\eta_i(s) e_i(s)=\sum_i\eta_i(s)\exp(-2\pi i \tau \Gamma) f_i$$
where the $\eta_i$ are local holomorphic functions.
In this case, the corresponding limit of the twisted period is $\Qb_\infty:=\sum_i \eta_i(0) \eb_i$.

We denote by $I$ the intersection form on $R^{n}f_{\ast}^{\times}\bb C$ and by $C$ the Weil operator so that 
$I(C v,\overline{v})$ is positive.
As the coefficients $\eta_i$ are holomorphic
$$\int_{X_{s}} i^{n^2}\eta(s)\wedge\overline{\eta(s)}=I(C\eta(s),\overline{\eta(s)})
=I\left(C e^{2\pi i\tau \Gamma}Q_\infty(s),\overline{e^{2\pi i\tau \Gamma}Q_\infty(s)}\right)
(1+O(|s|)).$$
By the $SL(2)$-orbit theorem~\cite[Theorem 6.6]{schmid}, that gives the asymptotic of the orbit of elements in $W_\bullet$,
the leading contribution comes from elements in $W_{\ell}$ not in $W_{\ell-1}$ :
\begin{eqnarray*}
I\left(C e^{2\pi i\tau \Gamma}Q_\infty(s),\overline{e^{2\pi i\tau \Gamma}Q_\infty(s)}\right)
&=&I\left(C e^{2\pi i\tau \Gamma}Q^+(s),  \overline{e^{2\pi i\tau \Gamma}Q^+(s)}\right)(1+0(Im(\tau)^{-1}))\\
&=&|s|^{-2\lambda}I\left(C e^{2\pi i\tau N} Q^+,\overline{e^{2\pi i\tau N} Q^+}\right)(1+0(Im(\tau)^{-1})).
\end{eqnarray*}
Now, for the principal nilpotent orbit $\eta^{+}(s):=e^{2\pi i\tau N} Q^+$, the quantity
\begin{eqnarray*}
I(C\eta^{+}(s),\overline{\eta^{+}(s)})
&=&I(C e^{2\pi i\tau N}Q^+,\overline{e^{2\pi i\tau N}Q^+})=I(C e^{2iIm(\tau) N}Q^+,\overline{Q^+})
\end{eqnarray*}
is a polynomial $P(\mathrm{Im}(\tau))$ of degree $\mu$ in $\mathrm{Im}(\tau)$, whose leading term is
$i^{n^2}\frac{(2i)^\mu \mathrm{Im}(\tau)^\mu}{\mu !}I(C N^\mu Q^+,\overline{Q^+})$.
The degree $\mu$ is the order of the nilpotent operator $N$ acting on $\Qb^+$. Hence, by the polarized condition~\cite[2.10)]{Cattani-kaplan}, and because $\Qb_\infty$ and $\Qb^+$ differ from an element in $W_{\ell-1}$,
it is exactly the order of the nilpotent monodromy operator $N$ acting on the limit twisted period $\Qb_\infty$.

The asymptotic of the $L^2$ norm is therefore
$$-\log\n\eta(s)\n^2_{L^2}\simeq \lambda\log|s|^2-\mu\log|\log|s|^2|.$$
\end{proof}

\begin{remark}
In the unipotent case, and with the notations as in the proof of the proposition,  from $\n\eta(s)\n^2_{L^2}=P(-\frac{1}{2\pi}\log|s|)+\rho_1(\tau)$
 we infer that the curvature of $(\pi_\ast(K_{\xc/S}),L^2)$
 (i.e. the Weil-Petersson metric) has Poincar\'e growth
\begin{eqnarray*}
\hess\log\n \eta(\tau)\n^2_{L^2}
=\left(\frac{(P')^2-PP''+\rho_2(z)}{P^2+\rho_3(\tau)}\right)id\tau\wedge d\overline{\tau}
\simeq \left(\frac{\mu}{(\mathrm{Im} \tau)^2}+\rho_4(\tau)\right)id\tau\wedge d\overline{\tau},
\label{limit}
\end{eqnarray*}
where the $\rho$'s are functions which, together with all their derivatives,
exponentially decrease to zero as $\mathrm{Im}(\tau)$ tends to $+\infty$, with rate of decay independent of $\mathrm{Re}(\tau)$.
\end{remark}

The above proposition allows us to determine some cases when $\alpha=0$, even though we may not have semi-stable reduction. For this, recall that $(X,x) \to (\Delta, 0)$, for $x\in X$, is a ordinary quadratic singularity if locally on $X$ the map can be written as a germ of a holomorphic function $f: (\bb C^n,0) \to (\bb C,0)$, so that $0$ is an isolated singularity of the level set $f=0$, and the Hessian of $f$ at $0$ is invertible. Such singularities can all be diagonalized to the form $\sum z_i^2=0$. 

\begin{prop}\label{quadraticordinary}
Suppose that $n \geq 2$ and  $f\colon X\to S$ has only ordinary quadratic singularities in $X_{0}$. Then $\alpha=\beta =0$.
\end{prop}
\begin{proof} 

We apply Proposition \ref{prop:lct}, and for this we need to find a normal crossings model. After a change of variables, each singularity $(X,x_k) \to \Delta$ is of the form $f = \sum z_i^2$. We blow up the point $(0, \ldots, 0)$, which can be described by imposing $t_j z_i = t_i z_j$, $(z_1, \ldots, z_n) \times [t_1, \ldots, t_n] \in \bb C^n \times \mathbb P^{n-1}$. Setting $t_1=1$ we find that the equation is of the form $z_0^2 \sum t_i^2$. We find that after blowing up each singular point we obtain that the special fiber is a NCD, of the form $\widetilde{X_0} + \sum 2E_k$ where $\widetilde{X_0}$ is the strict transform of $X_0$ under the blowup, and $E_k$ are the various exceptional divisors. On the other hand, blowing up a smooth point $x_i \in X$ the canonical bundle changes as $K_{X'} = K_{X} + nE_k$, so that $B= \sum nE_k$. We hence find that $c_{X_{0}}(f)= \inf\{1, (n+1)/2\} = 1$, and then $\alpha=0$. Moreover, the degeneracy index is, since $n \geq 2$,  equal to 1 so that $\beta =0$. \\

\end{proof}

\begin{remark}\label{rem:quadraticsingularity}
The proposition implies that if $f$ only has isolated ordinary quadratic singularity, both coefficients $\alpha$ and $\beta$ vanish, and hence the $L^{2}$ metric is continuous. 
\end{remark}



\section{Degeneration of the Quillen metric}\label{section:Quillen}
\subsection{Background on Quillen metrics}\label{quillenbackground}
\subsubsection{Grothendieck--Riemann--Roch in codimension 1} Let $f\colon\xc\to S$ be a smooth projective morphism of complex algebraic manifolds. Let $\vc$ be a vector bundle on $\xc$. The Grothendieck--Riemann--Roch theorem with values in Chow groups is an identity of characteristic classes
\begin{displaymath}
	\Ch(Rf_{\ast}\vc)=f_{\ast}(\Ch(\vc)\Todd(T_{\xc/S}))\in A_{\ast}(S)_{\mathbb{Q}}.
\end{displaymath}
We denoted by $A_{\ast}(S)$ Fulton's intersection theoretic Chow groups \cite{Fulton}. The relation is also valid in de Rham cohomology. In this section we focus on the "codimension one part" of the Grothendieck--Riemann--Roch formula. With values in Chow groups, this is written
\begin{displaymath}
	c_{1}(Rf_{\ast}\vc)=f_{\ast}(\Ch(\vc)\Todd(T_{\xc/S}))^{(1)}.
\end{displaymath} 
 The first Chern class of $Rf_{\ast}\vc$ equals the first Chern class of the determinant of the cohomology $\det Rf_{\ast}\vc$, also denoted $\lambda(\vc)$. It can be defined by the theory of Knudsen-Mumford \cite{KnudsenMumford}. Contrary to the individual relative cohomology groups, it is compatible with base change.

\subsubsection{Quillen metrics and the curvature formula} Suppose for simplicity that $\xc$ admits a K\"ahler metric on $X$, with K\"ahler form $\omega$, that we fix once and for all. If $\vc$ is equipped with a smooth hermitian metric $h$ and $T_{\xc/S}$ with the restriction of the K\"ahler metric, then the Grothendieck--Riemann--Roch formula in codimension 1 can be lifted to the level of differential forms. This is achieved by means of Chern-Weil theory and the theory of the Quillen metric. 

Let us briefly recall the definition of the Quillen metric. Let $s\in $S, and consider the fiber of $\lambda(\vc)$ at $s$:
\begin{displaymath}
	\lambda(\vc)_{s}=\bigotimes_{p}\det H^{p}(X_{s},\vc\mid_{X_{s}})^{(-1)^{p}}.
\end{displaymath}
By Hodge theory, and depending on the hermitian metric $h$ and the K\"ahler form $\omega$ restricted to $X_{s}$, the cohomology groups $H^{p}(X_{s},\vc\mid_{X_{s}})$ carry $L^{2}$ type metrics (using the Dolbeault resolution and harmonic representatives). Hence, $\lambda(\vc)_{s}$ has a induced metric that we still call $L^{2}$-metric, and that we write $h_{L^{2},s}$. This family of metrics is in general not smooth in $s$, due to possible jumps in the dimensions of the cohomology. Let $T(s)$ be the holomorphic analytic torsion attached to $(\vc,h)$ and $(T_{\xc/S},\omega)$:
\begin{displaymath}
	T(s)=\sum_{p=0}^{n}(-1)^{p}p\log\det\Delta^{0,p}_{s}.
\end{displaymath}
Here, we denoted by $\Delta^{0,p}_{s}$ the $\overline{\partial}$-laplacian acting on $A^{0,p}(\vc\mid_{X_{s}})$ ($(0,p)$ forms on $X_{s}$ with values in $\vc\mid_{X_{s}}$), and depending on the fixed hermitian data. Also, $\det\Delta^{0,p}_{s}$ denotes the zeta regularized determinant of $\Delta^{0,p}_{s}$ (restricting to strictly positive eigenvalues). The Quillen metric on $\lambda(\vc)_{s}$ is defined by
\begin{displaymath}
	h_{\quillen,s}=(\exp{T(s)})h_{L^{2},s}.
\end{displaymath}
This family of metrics is smooth in $s$. The resulting smooth metric on $\lambda(\vc)$ is called the Quillen metric, and we write $h_{Q}$ to refer to it. Observe that while the $L^{2}$-metric is defined using only harmonic forms (hence 0 eigenvalue for the laplacians), the Quillen metric involves the whole spectrum of the Dolbeault laplacians.

The curvature theorem of Bismut--Gillet--Soul\'e \cite{BGS1,BGS2,BGS3} is the equality of Chern--Weil differential forms on $S$
\begin{displaymath}
	c_{1}(\lambda(\vc),h_{\quillen})=f_{\ast}(\Ch(\vc,h)\Todd(T_{\xc/S},\omega)).
\end{displaymath}
By taking cohomology classes, one reobtains the Grothendieck--Riemann--Roch formula in de Rham cohomology.

\subsubsection{The Quillen metric close to singular fibers}
As a matter of motivation, we now review Yoshikawa's\cite{yoshikawa} results on the degeneration of the Quillen metric in a slightly simplified form.

Let $f\colon\xc\to S$ be a generically smooth, flat and projective morphism of complex algebraic manifolds. Therefore, with respect to the previous setting, we allow for singular fibers. We assume $S$ is one-dimensional and $f$ ha s a unique singular fiber. Recall that the Gauss map from the regular locus of $f$ to the space $\pb(T\xc)$ of rank one quotients
of $T\xc$
$$\begin{array}{cccc}
 \mu:& \xc-\Sigma_f&\longrightarrow&\pb(T\xc)\\&x&\longmapsto&T_x\xc/\ker df_x
\end{array}$$
described in coordinates through the isomorphism of $\pb(T\xc)$ with the space
$P(\Omega_\xc\otimes TS)=\pb(T\xc\otimes\Omega_S)$ of lines in $\Omega_\xc\otimes TS$
$$\begin{array}{cccc}
 \nu:& \xc-\Sigma_f&\longrightarrow&P(\Omega_\xc\otimes TS)\\&x
&\longmapsto&\left[\sum_{i=0}^n\frac{\partial s\circ f}{\partial z_i}(x)dz_i\otimes\frac{\partial}{\partial s}\right]
\end{array}$$
where $(z_i)$ is a local coordinate system on $\xc$ and $s$ is a local coordinate on $S$.
Consider the ideal sheaf $\ic_{\Sigma_f}:=\left(\frac{\partial s\circ f}{\partial z_i}(x)\right)$
on $\xc$ locally generated by the coefficients of $df$.
We resolve the singularities of $\mu$ and $\nu$ seen as a meromorphic map on $\xc$ by blowing up
the ideal $\ic_{\Sigma_f}$. Let $\XT\stackrel{q}\to\xc$ be any desingularization of the blowup of this ideal, and $E$ its exceptional divisor. We have a diagram
$$\begin{array}{ccc}
\xymatrix{&\pb(T\xc)\ar@<1ex>[d]^p\\ \XT\ar[r]^q\ar[ur]^\mut&\xc\ar@<1ex>@{-->}[u]^\mu}&
\xymatrix{&P(\Omega_\xc\otimes TS)\ar@<1ex>[d]^p\\ \XT\ar[r]^q\ar[ur]^\nut&\xc\ar@<1ex>@{-->}[u]^\nu}&
\xymatrix{&[\tau_i]=[\frac{\partial s\circ f}{\partial z_i}(x)]\ar@{|->}@<1ex>[d]^p\\ \xt\ar@{|->}[r]^q\ar@{|->}[ur]^\nut&x\ar@<1ex>@{|-->}[u]^\nu}
\end{array}$$
By construction, we see that
$\nut^\ast\oc_{T\xc\otimes f^\ast\Omega_S}(1)=\oc_\XT(-E)$.
Together with the isomorphism \linebreak $\pb(T\xc)\to\pb(T\xc\otimes\Omega_S)$, this gives for the resolution $\mut$ of $\mu$

\begin{equation} \label{starstar}
\mut^\ast\oc_{T\xc}(1)= q^\ast f^\ast TS\otimes\oc_\XT(-E).
\end{equation}
The tautological exact sequence on $\pb(T\xc)$ hence pulls back on $\XT$ to
\begin{eqnarray}\label{sigma}
0\to\mut^\ast U\longrightarrow q^\ast T\xc\stackrel{q^\ast df}\longrightarrow q^\ast f^\ast TS\otimes\oc_\XT(-E)\longrightarrow 0
\end{eqnarray}
where $U$ denotes the tautological hyperplane subbundle. With these preliminaries at hand, we can now state:

\begin{theo}[Yoshikawa \cite{yoshikawa}]\label{theorem:yoshikawa-quillen}
Fix a K\"ahler metric $h_\xc$ on $\xc$.
Let $(\vc,h)$ be a hermitian vector bundle on $\xc$. On the smooth locus, equip the determinant line bundle $\lambda(\vc)$ with the corresponding Quillen metric.
\begin{enumerate}
\item Let $\sigma$ be a local holomorphic frame for $\lambda(\vc)$ near the singular point $s=0$. Then $$\log \n \sigma\n^2_\quillen =\left(\int_E\Todd \mut^\ast U\frac{\Todd\oc_\XT(-E)-1}{c_1(\oc_\XT(-E))}
q^\ast\Ch(\vc)\right)\log |s|^2 +R(s)\quad\text{as}\quad s\to 0,$$
where $R(s)$ is a continuous function of $s$.
\item The curvature current is given, in a neighborhood of $s=0$, by
\begin{displaymath}
    \begin{split}
    c_{1}(\lambda(\vc), h_{\quillen})
    =&f_{\ast}(\Ch(\vc,h)\Todd(T_{\xc/S},h_{X}))^{(1,1)}\\
    &-\left(\int_E\Todd \mut^\ast U\frac{\Todd\oc_\XT(-E)-1}{c_1(\oc_\XT(-E))}
q^\ast\Ch(\vc)\right)\delta_{0},
\end{split}
\end{displaymath}
where the first term on the right of the equality is $L^{p}_{\mathrm{loc}}(S)$ for some $p>1$, and $\delta_{0}$ is the Dirac current at 0.
\item Denote by $\kappa$ minus the coefficient of the logarithmic singularity. Then the Quillen metric uniquely extends to a good hermitian metric on the $\Q$-line bundle $\lambda(V)\otimes\calo(-\kappa\cdot [0])$.
\end{enumerate}
\end{theo}
\begin{remark}\label{rem:good}
The third claim in the theorem is only implicitly stated in \cite{yoshikawa}. In fact, it is proven that the potential of the curvature current of the hermitian metric in (c) is of the form $\varphi(t) + \phi(t)$. Here $\varphi$ is smooth and $\phi$ is a finite sum of functions of the form $|s|^{2r}(\log|s|)^{k} g(t)$, where $r\in\Q\cap (0,1]$,  $k\geq 0$ is an integer and $g$ is smooth. This function and its derivatives satisfy the estimates in the definition of a good metric in the sense of Mumford \cite{mumford}.
\end{remark}
\subsection{The Nash blowup and the Yoshikawa class}
We proceed to develop an intersection theoretic approach to Yoshikawa's theorem. Instead of the theory of the Gauss map and the resolution of the Jacobian ideal, we introduce the Grassmanian scheme and the Nash blowup. Throughout we use the intersection theory of Fulton \cite{Fulton}. The advantage of our constructions is that they naturally exhibit a functorial behavior and allows for a better understanding of the topological term in Theorem \ref{theorem:yoshikawa-quillen} (cf. Definition \ref{def:yoshikawaclass}). 
We recover and expand concrete computations of Yoshikawa.

Let us say a word about the category where we place our arguments. We work in the category of schemes, mostly to be in conformity with the literature. However, the relevant arguments should be applicable in the analytic category, using relative singular cohomology instead of bivariant Chow groups.

\subsubsection{On the Jacobian cone}
Let $f : \xc\to S$ be a projective, flat, generically smooth morphism of integral noetherian schemes over $\mathbb{C}$, of relative dimension $n$.

Define the Jacobian ideal $\jac(\xc/S)$ as the annihilator of $\Lambda^{n+1}\Omega_{\xc/S}$. Assume that $\xc$ is locally a hypersurface  in a $S$-smooth scheme $\yc$ of dimension $n+1$. This is the case of hypersurfaces in $\mathbb{P}^{N}_{S}$, but also the case when $\xc$ and $S$ are smooth over $\bb C$ and $S$ is one-dimensional (consider the graph of the morphism). Locally on $\xc$, we have an exact sequence
\begin{eqnarray}\label{conormal}
0\to\ic_\xc/\ic^2_\xc\stackrel{d}{\to} {\Omega_{\yc/S}}_ {\mid \xc} \to \Omega_{\xc/S}\to 0
\end{eqnarray}
where the ideal $\ic_\xc$ of $\xc$ in $\yc$ is generated by an element $F$.
If one chooses (\'etale) local coordinates $y_0,\cdots,y_{n}$ on $\yc$ then  $\jac(\xc/S)$ is the $\oc_\xc$-ideal generated by
$\frac{\partial F}{\partial y_j}, j =0, \ldots, n $. Observe that this is, by definition, the first Fitting ideal of $\Omega_{\xc/S}$. This local description shows that the Jacobian ideal is indeed the ideal defining the singular locus of the structure morphism $f$. For example, if $f : \hc \to \check{\mathbb P}^N_{\cb}$ is the tautological family of hyperplane sections in some smooth complex projective variety $X$, then the Jacobian ideal just corresponds to the scheme parametrizing singular sections.

\subsubsection{On the Nash blowup}
We still work locally on $\xc$. Locally, we denote by $\yc$ a smooth $S$-scheme containing $\xc$ as a hypersurface. Let $\op{Gr}_n(\Omega_{\yc/S})$ be the Grassmannian of rank $n$-quotients of $\Omega_{\yc/S}$
and let $\xc\dashrightarrow \op{Gr}_n{\Omega_{\yc/S}}$ be the rational map defined by $x \mapsto (x, \Omega_{\xc/S, x})$,
called \emph{the Gauss map}.
The schematic closure $\widehat{\xc}$ of the image of this morphism is by definition \emph{the Nash blowup}
 of $\Omega_{\xc/S}$
and has the universal property that an $S$-morphism $t: \tc \to \widehat{\xc}$,
such that no component of $\tc$ has image contained in $V(\jac(\xc/S))$,
corresponds to a surjection $\Omega_{\xc_\tc/\tc} \to \mathcal E$, where $\mathcal E$ is locally free of rank $n$ on $\xc_\tc$. Denote by $\widehat{n}:  \widehat{\xc} \to \xc$ the obvious map.
As $\op{Gr}_n(\Omega_{\xc/S})$, understood as a Quot-scheme, is a closed subscheme of $\op{Gr}_n(\Omega_{\yc/S})$,
an equivalent definition, independent of  the choice of the ambient space $\yc$,
 is given by the closure of the $\xc/S$-smooth locus in $\op{Gr}_{n}(\Omega_{\xc/S})$. These constructions are summarized in the following diagram:


$$\xymatrix{
\widehat{\xc}\ar@{^{(}->}[rr]\ar[d]_{\widehat{n}}&&\Gr_n (\Omega_{\xc/S})\ar@{^{(}->}[rr]\ar[dll] &&\Gr_n(\Omega_{\yc/S})\ar[d]\\ 
\xc\ar[drr]\ar@{-->}[urrrr]^{Gauss} &\ar@{^{(}->}[rrr]&&& \yc\ar[dll]\\
&&S&&
}$$
This gives another interpretation of the Gauss map, considered by Yoshikawa. Actually,
suppose that $f: \xc \to S$ is a morphism of complex analytic manifolds, with $S$ of dimension one. Consider then the graph $\Gamma_f: \xc \to S \times \xc$.
Then the projection on $S$ from $\yc = S \times \xc$ is smooth, and the map $\xc^{sm} \subseteq \op{Gr}_n {\Omega_{\xc/S}} \to \op{Gr}_n{\Omega_\xc}$
from the $f$-smooth locus is given by $x \mapsto \left[\Omega_{\xc \times S/S, x} = \Omega_{\xc,x} \twoheadrightarrow \Omega_{\xc/S, x}\right]$.
This is simply a dual version of the usual Gauss map.
\subsubsection{Comparison with the resolution of the Jacobian ideal}
The Grassmanian construction, namely the Nash blowup, and the blowup of $\xc$ along the Jacobian ideal, actually coincide. This is useful in that both properties of blowups (structure of the exceptional divisor) and Grassmanians (existence of a universal locally free quotient and functoriality) can be simultaneously used.
\begin{lem}(see also~\cite{piene})
If $\xc$ is locally a hypersurface in an $S$-smooth scheme, then the blowup of $V(\jac(\xc/S))$ in $\xc$ is the Nash blowup of $\xc$.
\end{lem}
\begin{proof}
Denote by $b: \xc'\to\xc$ the blow up of $\xc$  along $\zc:=V(\jac(\xc/S))$
and $\widehat{n} : \widehat{\xc}\to\xc$ the Nash blowup of $\Omega_{\xc/S}$.
To construct a morphism from $\xc'$ to $\widehat{\xc}$, we have to construct a rank $n$ locally free quotient
of $\Omega_{\xc'\times_S\xc/\xc'}=b^\ast \Omega_{\xc/S}$. It is enough to show that the Gauss map locally  extends to $\xc'$,
since local extensions are separated hence unique. Locally, the ideal $\ic_\xc$ of $\xc$ in some smooth $S$-scheme $\yc$
is defined by an equation $F$ in $\oc_\yc$. Locally on $\xc'$, the ideal $ b^\ast\jac(\xc/S)$ is a free ideal $\oc_{\xc'}(-E)$
 generated by an element $u$ which is not a zero divisor.
The differential $b^\ast dF$ can then be written $u V$ for a uniquely determined nowhere vanishing section $V$ in $b^\ast \Omega_{\yc/S \mid \xc} $.
 From the sequence~\eqref{conormal} and the equality $V=``\frac{b^\ast dF}{u}"$, we infer
 \begin{eqnarray}\label{o1}
0\to(b^\ast\ic_\xc/\ic^2_\xc)\otimes\oc_{\xc'}(E)\stackrel{d\otimes 1}{\longrightarrow}  b^\ast {\Omega_{\yc/S}}_{\mid\xc}\to b^\ast  {\Omega_{\yc/S}}_ {\mid \xc} /V
\end{eqnarray}
 that gives a locally well-defined locally free quotient
 $b^\ast \Omega_{\xc/S}\to b^{\ast}{\Omega_{\yc/S}}_{\mid \xc} /V$.

To construct a morphism from $\widehat{\xc}$ to $\xc'$, by the universal property of blowing-up, we have to show that the Jacobian ideal $ \jac(\xc/S)$ becomes locally principal on $\widehat{\xc}$.
Consider the following diagram on $\widehat{\xc}$,
where the bottom line comes from the tautological sequence on $\op{Gr}_n(\Omega_{\yc/S})$,
the middle line comes from \eqref{conormal}, $M$ is the kernel of the rank $d$ quotient $\widehat{n}^\ast\Omega_{\xc/S}\to\qc$,
and $C$ the fiber product of $\widehat{n}^\ast\Omega_{\yc/S}$ and $M$ over $\widehat{n}^\ast\Omega _{\xc/S}$ :
\[
  \xymatrix{
    &         &          &0 \ar[d]& \\
    &  \widehat{n}^\ast \ic_\xc/\ic^2_\xc\ar[r] \ar@{=}[d]& C \ar[r]\ar[d] &M\ar[d] \ar[r]&0 \\
  & \widehat{n}^\ast \ic_\xc/\ic^2_\xc\ar[r]\ar[d]&\widehat{n}^\ast{\Omega_{\yc/S}}_{\mid\xc}\ar[r]\ar@{=}[d] &\widehat{n}^\ast\Omega _{\xc/S}\ar[r] \ar[d] & 0\\
    0\ar[r]& \nc\ar[r] & \widehat{n}^\ast{\Omega_{\yc/S} }_{\mid\xc}\ar[r] &\qc \ar[r] \ar[d] & 0\\
    &  &  & 0. & \\
  }
\]
We infer an induced map $C\to\nc$. As $C$ is a fiber-product, a diagram chasing provides an inverse map $\nc\to C$, so that $C$ is necessarily an invertible sheaf.
The sheaf $\qc$ being locally free, the Fitting ideal of $\widehat{n}^\ast\Omega _{\xc/S}$ is that of $M$, that is locally generated by the coefficient of the map $\widehat{n}^\ast \ic_\xc/\ic^2_\xc\to C$ between two invertible sheaves.
By functoriality of Fitting ideals, the pull back by $\widehat{n}$ of the Jacobian ideal is locally principal.
The two constructed maps are inverse over $\xc$ to each other, so that we can identify
$b: \xc'\to\xc$ and $\widehat{n} : \widehat{\xc}\to\xc$.
\end{proof}
Thanks to the lemma, on the blow-up $\xc'$ of $\xc$ along the Jacobian ideal there is a universal locally free quotient $b^{\ast}\Omega_{\xc/S}\to\qc$ (coming from the Grassmannian interpretation). We now consider its kernel. Let $E$ be the exceptional divisor of the blowup $b:\xc'\to\xc$, giving rise to the Cartesian diagram \\
$$ \xymatrix{ E \ar[d]^b \ar[r]^{i} &  \xc' \ar[d]^{b} \\
\zc \ar[r]^{i_\zc} & \xc.}$$
In the following lemma $L^{i}f^{\ast}$ is the $i$-th left derived inverse image under a morphism $f$. Recall that it is the sheaf defined by taking the $i$-th cohomology of the pull-back by $f$ of a local free resolution. Note that the sheaf $\Omega_{\xc/S}$ admits local free resolutions by the local hypersurface hypothesis. The lemma is to be compared with (the dual of) \eqref{starstar} restricted to $E$.
\begin{lem}\label{lemma:L_E}
Let $L_{E}$ be the kernel of the universal locally free quotient $b^{\ast}\Omega_{\xc/S}\to\qc$. Then $L_{E}$ is a locally free sheaf of rank $1$ on $E$. There is a canonical isomorphism
\begin{displaymath}
	L_{E}\simeq b^{\ast} L^{1} i_{\zc}^{\ast}\Omega_{\xc/S}\otimes\calo_{\xc'}(E)_{\mid E}.
\end{displaymath}
Furthermore, if $f\colon X\to S$ is a morphism of smooth algebraic varieties, then  $L_{E}\simeq\calo(E)_{\mid E}$. 

\end{lem}
\begin{proof}
That $L_{E}$ is supported on $E$ is immediate by construction. From the proof of the previous lemma, locally on $\xc$, there is a diagram of exact sequences
\[
  \xymatrix{
    &    0 \ar[d]      &   0 \ar[d]      &0 \ar[d]& \\
0 \ar[r]    &\kc  \ar[r]^{\hspace{-1.5cm}\alpha}\ar[d]	& (b^\ast\ic_\xc/\ic^2_\xc)\otimes\oc_{\xc'}(E) \ar[r]\ar[d] ^{d\otimes 1}&L_E\ar[d] \ar[r]&0 \\
  & b^\ast \ic_\xc/\ic^2_\xc\ar[r]^{d}&b^\ast(\Omega_{\yc/S\mid\xc}) \ar[r]\ar[d] &b^\ast\Omega _{\xc/S}\ar[r] \ar[d] & 0\\
& & \qc \ar[d]\ar@{=}[r]  &\qc \ar[d] & \\
    &  & 0 & 0. & \\
  }
\]
Because the differential $d:\ic_{\xc}/\ic_{\xc}^{2}\to\Omega_{\yc/S}$ vanishes on $\zc$ the induced map $b^{\ast}\ic_{\xc}/\ic_{\xc}^{2}\to b^{\ast}\Omega_{\yc/S}$ vanishes on $E$ as well. Moreover the morphism $d\otimes 1$ remains injective after restricting to $E$. It follows that $\alpha_{\mid E}$ is vanishes identically, and hence there is an isomorphism
\begin{equation}\label{eq:L_E-1}
	L_{E}\simeq (b^\ast\ic_\xc/\ic^2_\xc)_{\mid E}\otimes\oc_{\xc'}(E)_{\mid E}.
\end{equation}
This shows that $L_{E}$ is locally free of rank $1$. Now we claim that there is an isomorphism
\begin{equation}\label{eq:L_E-2}
	(b^\ast\ic_\xc/\ic^2_\xc)_{\mid E}\simeq b^{\ast}L^{1}i_{\zc}^{\ast}\Omega_{\xc/S}.
\end{equation}
First of all, it is clear that $(b^\ast\ic_\xc/\ic^2_\xc)_{\mid E}=b^{\ast} i_{\zc}^{\ast}(\ic_\xc/\ic^{2}_{\xc})$. Second, from the exact sequence \eqref{conormal} we derive
\begin{displaymath}
	L^{1}i_{\zc}^{\ast}\Omega_{\xc/S}\simeq i_{\zc}^{\ast}(\ic_\xc/\ic^2_\xc).
\end{displaymath}
The claim follows. Hence \eqref{eq:L_E-1}--\eqref{eq:L_E-2} give raise to an isomorphism as in the statement. One can check it does not depend on the (local) choice of $\yc$, so that it is a canonical isomorphism and globalizes. This completes the proof of the first claim.

For the second assertion, it is enough to specialize the previous argument with $Y=X\times S$. In this case, it is immediate that
\begin{displaymath}
    \ic_{X}/\ic_{X}^{2}=f^{\ast}\Omega_{S}.
\end{displaymath}
Since $f$ is generically smooth and $S$ is one-dimensional, the singular locus of $f$ in $S$ is zero dimensional. We thus see that 	$$L^{1}i_{\zc}^{\ast}\Omega_{\xc/S}\simeq i_{\zc}^{\ast}(\ic_\xc/\ic^2_\xc)=(f\circ i_{Z})^{\ast}(\Omega_{S})$$
is a trivial line bundle. 
\end{proof}

\subsubsection{The Yoshikawa class}
The previous notations regarding the morphism $f\colon\xc\to S$ are still in force. In particular, the singular locus is a closed subscheme $\zc$ in $\xc$, the Nash blowup along $\zc$ is $b\colon\xc'\to\xc$ and $E$ is the exceptional divisor. We now digress on localized characteristic classes in the theory of Chow groups. This formalism, combined with the previous observations on Nash blowups, reveals useful to arrive to a conceptual explanation of the topological term in Yoshikawa's asymptotics. To be consistent with the literature on intersection theory and Chow groups (cf. Fulton's \cite{Fulton}, especially the relative setting of Chapter 20), from now we assume that $S$ is regular, for instance $\Spec R$ with $R$ a discrete valuation ring. Also, we will make extensive use of the theory of localized Chern classes. We refer the reader to \cite[Chap. 18.1]{Fulton} for the main construction of localized Chern classes of generically acyclic complexes, using the Grassmannian graph construction. We 
also cite \cite[Sec. 3]{Abbes} and \cite[Sec. 2]{Kato-Saito}, that recast the main properties of the localized Chern classes of generically acyclic complexes, in the form that will be used here. 

Recall that a bivariant class $c\in A(\xc\to\yc)$ is a rule that assigns, to every $\yc$-scheme, say $\yc'$, a homomorphism
\begin{displaymath}
    c\colon A_{\ast}(\yc')
    \longrightarrow A_{\ast}(\xc'),
\end{displaymath}
where $\xc'$ is the base change of $\xc$ to $\yc'$. This homomorphism is subject to several compatibilites (proper push-forward, flat pull-back and intersection product). We refer to \cite[Chap. 17]{Fulton} for the precise formulation of these.

Suppose we are given a multiplicative characteristic class $T$, corresponding to a power series \linebreak $T(x) \in 1 + x \bb Q[x]$. Thus, to a vector bundle $\mathcal{E}$ on $\xc$ it associates homomorphisms on Chow groups $T(\mathcal{E}):A_*(\xc)_\bb Q \to A_*(\xc)_\bb Q$, and to a bounded complex of vector bundles $\mathcal{E}^\bullet$ it associates the homomorphism $\prod T(\mathcal{E}^i)^{(-1)^i}$, compatible with pull-backs. Let $b: \xc' \to \xc$ be the Nash blowup of the morphism $f:\xc \to S$, with exceptional divisor $E$. On $\xc'$ there is the universal locally free quotient $b^* \Omega_{\xc/S} \to \qc$. Because $\xc$ is locally an hypersurface in a smooth $S$-scheme, this is quasi-isomorphic to a three term complex of vector bundles. It is acyclic off the exceptional divisor $E$. Thus, following \cite[Chap. 18.1]{Fulton}, there are  localized bivariant Chern classes $c_i^E(b^* \Omega_{\xc/S} \to \qc)\in A(E \to \xc'),\,i >0$. Consequently, the class $T(b^*\Omega_{\xc/S} \to \
 \qc) - 1 = T(b^* \Omega_{\xc/S}) T(\qc)^{-1} -1$ admits a refinement as a bivariant Chern class. Indeed, $T$ itself can be expressed as a power series in the Chern classes $c_{i}$, and the refinement to a bivariant class is obtained by replacing $c_i$ by $c_i^E$ in this power series representation. This refinement shall be denoted
\begin{displaymath}
    T^E(b^* \Omega_{\xc/S} \to \qc)\in A(E \to \xc')_{\bb Q},
\end{displaymath}
or simply $T^{E}$ to simplify the notations. If $[\xc']\in A_{\ast}(\xc')$ is the cycle class of $\xc'$, then $T^{E}$ sends $[\xc']$ into $A(E)_{\bb Q}$. The usual notation for this class is $T^{E}\cap [\xc']$. We will later be interested in the top degree terms of such classes.

The following lemma computes $T^{E}\cap [\xc']$ in terms of characteristic classes depending only on $\oc(E)$.
\begin{lem}\label{lemma:simplification-bivariant}
Assume that the base $S$ is one-dimensional. Then:
\begin{enumerate}
\item as a bivariant class $c_{1}(L^{1}i_{\zc}^{\ast}\Omega_{\xc/S})$ vanishes. In particular, we have an equality of bivariant classes
\begin{displaymath}
    c_{1}(L_{E})=c_{1}(\calo(E)_{|E}).
\end{displaymath}
\item The bivariant class $T^E$ satisfies the formula $$T^E \cap [\xc'] =  \left(\frac{T(\oc(E)_{\mid E})-1}{c_1(\oc(E)_{\mid E})} \cap [E] \right)$$ in $A_{*}(E)_{\bb Q}$.
\item $T^E$ also satisfies the formula $$T(\qc_{\mid E}) \left(T^E \cap [\xc']\right)= T\left(Li^* b^* \Omega_{\xc/S}\right)\left (\frac{1-T(\oc(E)|_E)^{-1} }{c_1(\calo(E)|_E)}\right) \cap [E].$$
\end{enumerate}
\end{lem}
\begin{proof} 
For the first item, under our running assumptions on $\xc$ (locally hypersurface hypothesis, $f$ generically smooth and $S$ one-dimensional and regular) the proof of \cite[Lemma 5.1.3]{Kato-Saito} can be adapted \emph{mutatis mutandis}. The equality $c_{1}(L_{E})=c_{1}(\calo(E)_{|E})$ the follows from Lemma \ref{lemma:L_E}.

For the second claim, by a deformation to the normal cone argument with respect to the closed immersion $E \to \xc'$, we can assume that $i:E \to \xc'$ is the section of a projection $p: \xc' \to E$. In this case, since $p_* i_* = \operatorname{Id}$, the direct image $i_*: A_*(E) \to A_*(\xc')$ is necessarily injective. Moreover, for any localized Chern class as in the statement, $$i_* (T^E(b^* \Omega_{\xc/S} \to  \qc)\cap[\xc']) =  (T(b^* \Omega_{\xc/S} \to \qc)-1)\cap[\xc'].$$  On $\xc'$ we have the tautological sequence,
$$0 \to L_E \to b^* \Omega_{\xc/S} \to \qc \to 0.$$
By Lemma \ref{lemma:L_E}, $L_{E}$ is a line bundle on $E$. Since $i:E\to\xc'$ is a retraction, the line bundle $L=p^{\ast}L_{E}$ on $\xc'$ extends $L_{E}$ and there is an exact sequence
$$0 \to L(-E) \to L \to L_E \to 0.$$
We thus have a quasi-isomorphism of complexes
\begin{displaymath}
	[\underset{-1}{L(-E)}\to\underset{0}{L}]\simeq [\underset{0}{b^{\ast}\Omega_{\xc/S}}\to\underset{1}{\qc}].
\end{displaymath}
Consequently
\begin{displaymath}
	\begin{split}
		(T(b^* \Omega_{\xc/S} \to \qc)-1)\cap[\xc']&=(T(L(-E)\to L)-1)\cap[\xc']\\
		&=(T(L)T(L(-E))^{-1}-1)\cap[\xc'].
	\end{split}
\end{displaymath}
The class $T(L) T(L(-E))^{-1}-1$ is naturally divisible by $c_1(\mathcal {O}(E))$. We can thus rewrite \begin{equation}\label{eq:T(L)T(L(-E))}
(T(L) T(L(-E))^{-1} -1)\cap[\xc'] = \frac{T(L) T(L(-E))^{-1}-1}{c_1(\mathcal {O}(E))}\cap [E].
\end{equation}
Finally, by Lemma \ref{lemma:L_E} we also know that $L_{E}=L^{1}i_{\zc}^{\ast}\Omega_{\xc/S}\otimes\calo (E)\mid_{E}$, and hence by the first item we infer $c_{1}(L_{E})=c_{1}(\calo(E)_{\mid E})$. Plugging this relation into \eqref{eq:T(L)T(L(-E))}, we arrive at the equality
$$T^E(b^* \Omega_{\xc/S} \to \qc) \cap [\xc'] = \frac{T(\calo(E)_{\mid E})-1}{c_1(\mathcal {O}(E)_{\mid E})} \cap [E],$$ as was to be shown. \\
The final claim follows the same lines (and notation) as the second, and the completely formal computations
\begin{eqnarray*} i_*\left(T(\qc)\cap (T^E - 1)\cap [\xc'] \right)& = & T(\qc)(T(b^* \mathcal \Omega_{\xc/S})T(\qc)^{-1} - 1)\cap [\xc'] \\
& = &  \left(T(b^* \mathcal \Omega_{\xc/S}) \cap (1-T(\qc)T(b^* \mathcal \Omega_{\xc/S})^{-1} \right) \cap [\xc']\\
& = &  T(b^* \mathcal \Omega_{\xc/S}) \cap \left(1-T(L(-E))T(L)^{-1}\right) \cap [\xc']. \end{eqnarray*}

\end{proof}

Recall that $\Td^*$ is the multiplicative characteristic class determined by $\frac{(-x)}{1-e^{-(-x)}} = \frac{x}{e^x - 1}$. We next define the Yoshikawa class, inspired by Theorem \ref{theorem:yoshikawa-quillen}.

\begin{definition}[Yoshikawa class] \label{def:yoshikawaclass} Keep the assumptions of the introduction of this chapter. Let $V$ be a vector bundle on $\xc$. Given a birational and proper morphism $\pi: \widetilde{\xc} \to \xc$ of integral schemes, with a surjection $\pi^* \Omega_{\xc/S} \to \mathcal E$, for some vector bundle $\mathcal{E}$ of rank $n$, define the Yoshikawa class as the cycle class $$\mathcal{Y}(\xc/S,V) = \Ch(i_{\zc}^{\ast}V)\cdot\pi_* (\Td^*(\mathcal E_{\mid D}) \Td^{\ast\,D}(\pi^* \Omega_{\xc/S} \to \mathcal E) \cap [\widetilde{\xc}])\in A_{\ast}(\zc)_{\mathbb Q},$$ where $D=\pi^{-1}(\zc)$. For the trivial sheaf, we denote it by $\mathcal{Y}(\xc/S)$.
\end{definition}
\begin{prop} [Independence]\label{prop:yoshi-indep} The Yoshikawa class is independent of the choice of birational morphism $\pi\colon\widetilde{\xc}\to\xc$ and surjection $\pi^{\ast}\Omega_{\xc/S}\to\mathcal{E}$. 
\end{prop}
\begin{proof}
The first assertion follows from the existence of the moduli of rank $d$-quotients of $\Omega_{\xc/S}$.
Indeed, any datum as in the statement can be compared to the universal case on the Nash blowup: there exists a morphism to the Nash blowup $\varphi\colon\widetilde{\xc}\to\xc'$ and a commutative diagram
\begin{displaymath}
		\xymatrix{
			&\varphi^{\ast}b^{\ast}\Omega_{\xc/S}\ar@{->>}[d]\ar@{=}[r]	&\pi^{\ast}\Omega_{\xc/S}\ar@{->>}[d]\\
			&\varphi^{\ast}\qc\ar[r]^{\sim}	&\mathcal{E},
		}
\end{displaymath}
where the left-most vertical arrow is induced from the universal surjection on the Nash blowup. Moreover, we observe that $$L\varphi^{\ast} b^{\ast}\Omega_{\xc/S}=\varphi^{\ast}b^{\ast}\Omega_{\xc/S}.$$ Indeed,  since $\xc$ is noetherian and is locally a hypersurface in an $S$-smooth scheme, $\Omega_{\xc/S}$ admits a two-term resolution by locally free sheaves $0 \to F_1 \to F_2 \to \Omega_{\xc/S}\to 0$. Notice that the pullback \linebreak $0 \to b^* F_1 \to b^* F_2 \to b^* \Omega_{\xc/S} \to 0$ is still exact, since the left-most map is generically injective on an integral scheme, and hence globally injective. Repeating the argument with $\varphi$, establishes the relationship. We can then invoke the very construction of the localized Chern classes and the projection formula \cite[p. 31, especially C2]{Abbes}. 
\end{proof}
Yoshikawa's theorem works with a smooth desingularization of the Gauss map. The above proposition hence proves:
\begin{cor}
Under the hypotheses of Theorem \ref{theorem:yoshikawa-quillen}, the degree of $\ \mathcal{Y}(\xc/S,\vc)$ is the coefficient of the logarithmic singularity of the Quillen metric.
\end{cor}
\begin{remark}
We expect that the hypothesis of smooth total space $X$ can be weakened with the same conclusion on the logarithmic singularity of the Quillen metric. This is one of the motivations of our treatment of the Yoshikawa class. 
\end{remark}
\begin{prop}[Functoriality]\label{prop:functoriality}
Suppose given a Cartesian diagram
\begin{displaymath} 
	\xymatrix{
		&\xc_{T}\ar[r]^{p'}\ar[d]^{f'}	&\xc\ar[d]^{f}\\
		&T\ar[r]^{p}	&S
	}  
\end{displaymath}
where $f'$ is a generically smooth morphism of integral schemes and $p\colon T\to S$ is a locally complete intersection morphism. 
Then ${p}^! \mathcal{Y}(\xc/S,V) = \mathcal{Y}(\xc_{T}/T, p^{\prime\ast}V)$, where $p^!$ denotes the refined Gysin morphism associated to $p$. 
\end{prop}
\begin{proof}
Let $\zc'$ be the Jacobian scheme of the morphism $f'$. By the functoriality of Fitting ideals, the scheme $\zc'$ is the base change of $\zc$ to $T$ and there is a canonical isomorphism $(\xc_{T})^{\prime}\to(\xc^{\prime})_{T}$ for the Nash blowups. In particular, it is legitimate to drop the parentheses in the notations. Factoring $T\to S$ as the composition of a smooth morphism and a regular closed immersion, we can treat each case separately. They are similar, but the smooth case is simpler so we suppose henceforth that $T\to S$ is a regular closed immersion of constant codimension $d$. Now, consider the cartesian diagrams
$$\xymatrix{\xc_T' \ar[r] \ar[d] & \ar[d] \xc' \\
\xc_T \ar[d] \ar[r]^p & \xc \ar[d] \\
T \ar[r] & S.}$$ 
Any bivariant class $T^E$ with respect to $E \to \xc$ satisfies $p^!(T^E \cap [\xc ]) = T^{E'} \cap p^{!}[\xc']$ (see \cite[Sec. 17.1, axiom (C3)]{Fulton})
and clearly $p^! [\xc'] = [\xc'_T ]$. Moreover, we have an induced cartesian diagram
$$\xymatrix{
E' \ar[r] \ar[d]^{\pi'} & E \ar[d]^{\pi} \\
\zc' \ar[r] \ar[d] & \zc \ar[d] \\
\xc_T \ar[r] & \xc .
}$$
Then as the refined Gysin maps commute with proper pushforward  \cite[Thm. 6.2]{Fulton}, \linebreak ${\pi'}_* \left(T^{E'} \cap [\xc'_T]\right)= p^! \pi_* \left(T^E \cap [\xc] \right)$. This implies the statement.

\end{proof}
\subsubsection{Computations of the Yoshikawa class} In the following proposition, we show that the Yoshikawa class can be written in terms of Segre classes (cf. \cite[Chap. 4]{Fulton}). In the particular case of isolated singularities and regular total space, the formula reduces to a classical topological invariant of those: the Milnor number. Recall that for a germ of an isolated hypersurface singularity $f\colon (\bb{C}^{n+1},0)\to(\bb C,0)$, the Milnor number is defined as
\begin{equation}\label{eq:Milnor}
    \mu_{0}=\dim_{\bb C}\ \frac{\bb{C}\lbrace{z_0,\ldots,z_n\rbrace}}{\left(\frac{\partial f}{\partial z_{0}},\ldots,\frac{\partial f}{\partial z_{n}}\right)}.
\end{equation}
The results are a cohomological refinement of Yoshikawa's formulas \cite{yoshikawa2, yoshikawa}.

\begin{prop}\label{cor:Milnor} 
Suppose that $S$ is one-dimensional and $b\colon\xc'\to\xc$ is the Nash blow-up with exceptional divisor $E$. Then:
\begin{enumerate}
    \item the Yoshikawa class fulfills the equality $$\mathcal{Y}(\xc/S)= \Td^*(i_{\zc}^* \Omega_{\xc/S}) \cap \sum_{k=0}^\infty\frac{(-1)^{k+1}}{(k+2)!}s_{n-k}(\zc),$$
where $s_{n-k}(\zc)=(-1)^{k}b_{\ast}(E^{k})\in A_{n-k}(\zc)$ is a Segre class.
    \item if $Y$ is a smooth projective variety, and $\xc\to S$ is a family of hypersurfaces in $Y\times S$, then
    \begin{displaymath}
        \mathcal{Y}(\xc/S)=\Td^{\ast}(\Omega_{Y\mid \zc})\cap\sum_{k=0}^\infty\frac{(-1)^{k+1}}{(k+2)!}s_{n-k}(\zc).
    \end{displaymath}
    \item if $\xc$ is regular, $\xc \to S$ is the germ of a morphism over a disk ($S=\Spec\bb C \lbrace t\rbrace$) and has only isolated singularities in the special fiber $\xc_{0}$, then $$\operatorname{deg} \mathcal{Y}(\xc/S)= \frac{(-1)^{n+1}}{(n+2)!}\sum_{x\in\xc_{0}} \mu_{\xc,x}$$
where $\mu_{\xc, x}$ denotes the Milnor number of the singularity at $x$.
\end{enumerate}
\end{prop}

\begin{proof} As in the proof of Proposition \ref{prop:yoshi-indep}, one can  show $$Li^* b^* \Omega_{\xc/S} \simeq Li^* Lb^* \Omega_{\xc/S} \simeq Lb^* Li_{\zc}^* \Omega_{\xc/S}.$$ Moreover, by Lemma \ref{lemma:simplification-bivariant} (a) and the observation $L^{j}i_{\zc}^{\ast}\Omega_{\xc/S}=0$ for $j\geq 2$ (since there exists a local free resolution of length 2 of $\Omega_{\xc/S}$), we conclude for the Chern classes the relation $c_{j}(Li_{\zc}^{\ast}\Omega_{\xc/S})=c_{j}(i_{\zc}^{\ast}\Omega_{\xc/S})$. With this understood, the first formula is a direct computation using the third claim in Lemma \ref{lemma:simplification-bivariant} and the projection formula.

For the second formula, by \eqref{conormal} applies with $Y\times S$ in place of $\yc$, we see that $i_{\zc}^{\ast}\Omega_{\xc/S}=\Omega_{Y\mid \zc}$.

For the third property, we can suppose that $f\colon\xc\to S$ has an isolated singularity at a single closed point $x$ in the special fiber $\xc_{0}$.  Furthermore $i^*_\zc \Omega_{\xc/S}$ is supported on a zero-dimensional space, and its Todd class is necessarily 1. We then have by the established formulas, 
\begin{displaymath}
    \deg \mathcal{Y}(\xc/S)=\frac{(-1)^{n+1}}{(n+2)!}\deg s_{0}(\zc).
\end{displaymath}
Now, because $\xc\to S$ is Cohen-Macaulay and the singularity is isolated, the degree of the Segre class $s_{0}(\zc)$ is computed by the colength of the Jacobian ideal \cite[Ex. 4.3.5 (c)]{Fulton}. This is the Milnor number as defined in  \eqref{eq:Milnor}.
\end{proof}
The following lemma will be useful in some computations with the Yoshikawa class. As an example of use, we refer to Theorem \ref{prop:vanishingcycleYoshikawa} and Theorem \ref{prop:yoshikawa-kulikov} below.
\begin{lem}\label{prop:bloch} Let $f\colon \xc \to S$ be a germ of a fibration over the unit disk, with regular total space $\xc$. Then
$$\deg c_n(\qc_{|E}) =\deg c_{n+1}^{\zc}(\Omega_{\xc/S}) \cap [\xc]  = (-1)^{n} \left(\chi(\xc_\infty) - \chi(\xc_0) \right),$$
where $\xc_\infty$ is a generic fiber and $\chi$ is the topological Euler characteristic.
\end{lem}
\begin{proof}
For the equality
\begin{displaymath}
    \deg c_{n+1}^{\zc}(\Omega_{\xc/S}) \cap [\xc]= (-1)^{n} \left(\chi(\xc_\infty) - \chi(\xc_0) \right),
\end{displaymath}
we observe that
\begin{displaymath}
    \deg c_{n+1}^{\zc}(\Omega_{\xc/S}) \cap [\xc]=
    \deg c_{n+1}^{\xc_{0}}(\Omega_{\xc/S})\cap [\xc]
\end{displaymath}
and then we refer to \cite[Example 14.1.5]{Fulton}. For the first equality, we recall from Lemma \ref{lemma:L_E} the tautological exact sequence on the Nash blowup $\xc'$
\begin{displaymath}
    0 \to L_E \to b^* \Omega_{\xc/S} \to \qc \to 0.
\end{displaymath}
By the Whitney formula for localized Chern classes \cite[Prop. 3.1 (b)]{Abbes} we have
\begin{equation}\label{eq:bloch-Q-1}
    c_{n+1}^{E}(b^{\ast}\Omega_{\xc/S})\cap [\xc']=c_{n}(\qc_{|E})(c_{1}^{E}(L_{E})\cap [\xc']).
\end{equation}
By the vanishing property in Lemma \ref{lemma:simplification-bivariant} (a), we also have
\begin{equation}\label{eq:bloch-Q-2}
    c_{1}^{E}(L_{E})\cap [\xc']=[E].
\end{equation}
From \eqref{eq:bloch-Q-1}--\eqref{eq:bloch-Q-2} we derive
\begin{equation}\label{eq:bloch-Q-3}
    c_{n+1}^{E}(b^{\ast}\Omega_{\xc/S})\cap [\xc']=
    c_{n}(\qc_{|E})\cap [E].
\end{equation}
To conclude, we apply the projection formula of localized Chern classes with respect to proper morphisms \cite[p. 31 $(C_1)$]{Abbes}, that implies
\begin{equation}\label{eq:bloch-Q-4}
    \deg c_{n+1}^{E}(b^{\ast}\Omega_{\xc/S})\cap [\xc']
    =\deg c_{n+1}^{\zc}(\Omega_{\xc/S})\cap[\xc].
\end{equation}
We complete the proof by combining \eqref{eq:bloch-Q-3}--\eqref{eq:bloch-Q-4}.
\end{proof}
\subsubsection{The Yoshikawa class for families of hypersurfaces}

Recall that the discriminant or dual variety of a smooth variety $Y \subseteq \mathbb P^N$ is a variety $\Delta_Y \subseteq \check{\mathbb P}^N$, parametrizing the hypersurfaces $H \in \check{\mathbb P}^N$ such that $Y \cap H$ is singular. Here $Y \cap H$ is regarded as a scheme. In many interesting cases $\Delta_Y$ is a hypersurface. Let us mention the case of the $d$-Veronese embedding, $\mathbb P^n \subseteq \mathbb P^N$. In this case $\Delta_Y$ parametrizes singular hypersurfaces of degree $d$ in $\mathbb P^n$. \\
We denote by $F: \mathcal H \to \check{\mathbb P}^N$ the universal family of hyperplane sections of $Y$. The $F$-singular locus $\mathcal Z \to \Delta_Y$ can be described as the projective bundle $\mathbb P(N_{Y/\mathbb{P}^{N}})$ over $Y$, where $N$ denotes the normal bundle of $Y \subseteq \mathbb P^N$. Indeed, a singular point in a hyperplane section is nothing but hyperplane $H$, a point $y \in Y \cap H$ such that $T_y H \subseteq T_y \mathbb P^n$ contains $T_y Y$, so that $H$ corresponds to a vector in $\mathbb P(N_{Y/\mathbb{P}^{N},y})$, the projectivised normal bundle of $Y \subseteq \mathbb P^n$ at $y$. Hence the $F$-singular locus is just the projectivised normal bundle of $Y \subseteq \mathbb P^N$ \cite{GKZ}, p. 27. In particular, $\Delta_Y$, being the image of $\mathbb P(N_{Y/\mathbb{P}^{N}})$ in $\check{\mathbb P}^N$, is irreducible.

\begin{theo}\label{prop:vanishingcycleYoshikawa} Suppose that $f\colon \xc \to S$ is a family of hyperplane sections of a smooth complex projective variety $Y \subseteq \mathbb P^N$ of dimension $n+1$, over a regular base $S$. Let $\zc$ be the singular scheme of $f$. Then the codimension $n+1$-component of \ $\mathcal{Y}(\xc/S)$ is given by
\begin{displaymath}
    \mathcal{Y}(\xc/S)^{(n+1)}=\frac{(-1)^{n+1}}{(n+2)!}c_{n+1}^{\zc}(\Omega_{\xc/S})\cap[\xc].
\end{displaymath}
Consequently,
\begin{displaymath}
    \deg \mathcal{Y}(\xc/S)=\frac{(-1)^{n+1}}{(n+2)!}\int_{\xc_0}c_{n+1}^{\xc_{0}}(\Omega_{\xc/S})\cap[\xc].
\end{displaymath}
\end{theo}
\begin{remark}
When $\xc$ is regular, one can see that $f$ has isolated singularities. Then, according to the above theorem and Lemma \ref{prop:bloch}, the degree of the Yoshikawa class is given by the change of Euler characteristics, or equivalently the vanishing cycles. This is compatible with Proposition \ref{cor:Milnor} (c), since the sum of the Milnor numbers equals the number of vanishing cycles. 
\end{remark}
\begin{proof} 
For the first point, by Proposition \ref{prop:functoriality}, and the analogous functoriality for $c_{n+1}^{\zc}(\Omega_{\xc/S})\cap[\xc]$, it is enough to prove that $$\mathcal{Y}(\xc/S)^{(n+1)} = \frac{(-1)^{n+1}}{(n+2)!} c_{n+1}^{\zc}(\Omega_{\xc/S}) \cap [\xc]$$ when $\xc \to S$ is the universal situation $\mathcal H \to \check {\mathbb P}^N$, with $\zc=\mathbb{P}(N)$.

We start by proving that $$[\zc] = c_{n+1}^\zc(\Omega_{\mathcal H/\check{\mathbb P}^N}) \cap [\mathcal H],$$ and later we will relate the Yoshikawa class to $[\mathbb P(N)]$. Consider the resolution $$\calo(-{\mathcal H})_{|\mathcal H} \to \Omega_{Y \times \check{\mathbb P}^N/ \check{\mathbb P}^{N} |\mathcal H}\to\Omega_{\mathcal H/\check{\mathbb P}^{N}}\to 0.$$ It determines a section $\sigma$ of $\Omega_{Y \times \check{\mathbb P}^N/ \check{\mathbb P}^N}({\mathcal H})_{|\mathcal H}$ whose schematic zero locus is $\zc$. This is of maximal codimension $n+1$ in $\mathcal H$, hence by \cite[Prop. 14.1 (c)]{Fulton} the corresponding localized Chern class is given by $[\zc]$. All in all, we conclude
\begin{displaymath}
    [\zc]=c_{n+1}^{\zc}(\Omega_{\mathcal{H}/\check{\mathbb{P}}^{N}}\otimes\calo(\mathcal{H})_{|\mathcal{H}})=c_{n+1}^{\zc}(\Omega_{\mathcal{H}/\check{\mathbb{P}}^{N}}),
\end{displaymath}
where the last equality is easily checked from the very construction of the localized Chern classes through the Grassmannian graph construction (see \cite[Sec. 3]{Abbes} and use that $\calo(\mathcal{H})_{|\mathcal{H}}$ is invertible, hence tensoring by it induces an isomorphism on grasmannians and does not alter the construction in \emph{loc. cit}).

Now we compute the $(N-1)$-dimensional component of the Yoshikawa class in the universal situation. First, we observe that the codimension $n+1$ component $\mathcal{Y}(\mathcal{H}/\check{\mathbb{P}}^{N})$ is concentrated on the $N-1$ dimensional irreducible subscheme $\zc$, and hence is a multiple thereof:
\begin{displaymath}
    \mathcal{Y}(\mathcal{H}/\check{\mathbb{P}}^{N})^{(n+1)}=m [\zc],
\end{displaymath}
for some rational number $m$. Second, we determine the coefficient $m$ by ``evaluating" on a point. For this, denote by $b: \mathcal H' \to \mathcal H$ denotes the Nash blowup. The induced map $E \to \zc$ has the structure of a projective bundle of rank $n$. As in the proof of Lemma \ref{lemma:simplification-bivariant}, write the Yoshikawa class as 
$$b_* \left(\Td^*(Li^* b^* \Omega_{\mathcal{H}/\check{\mathbb P}^N}) \cap \left( \frac{1-\Td^*(L_E(-E))\Td^*(L_E)^{-1}}{c_1(\calo(E)_{|E})}
\right) \cap [E]\right), $$ where $i$ is the closed immersion of $E$ into $\mathcal{H}'$. 
Let $k: p \to Z$ be any (closed) point of $\zc$,  necessarily a closed regular immersion of codimension $N-1$. Then we have a Cartesian diagram 
$$\xymatrix{\mathbb P^n \ar[r]^{k'} \ar[d]^{b'} & E \ar[d]^b \\
p \ar[r]^k & \zc.}$$
Then as $k^* [\zc] = [p]$, it is enough to compute $k^* \mathcal{Y}(\mathcal H/\check {\mathbb P}^N)$. We obviously have $${k'}^* b^*  L^1 {i_\zc}^* \Omega_{\mathcal H/\check {\mathbb P}^N} = {b'}^* k^*  L^1 {i_\zc}^* \Omega_{\mathcal H/\check {\mathbb P}^N}$$ is a trivial line bundle over a point. Therefore, by Lemma \ref{lemma:L_E} we find ${k'}^* L_E = \calo(E)_{|\mathbb P^n}= \calo(-1)$. Furthermore, $b_* k^* = {b'}_* {k'}^*$ and we conclude that the pullback of the Yoshikawa class is  given by $\int_{\bb P^{n}} \frac{1-\Td^*(\calo(-1))^{-1}}{c_1(\calo(-1))}.$ This further simplifies to $$m=\deg k^* \mathcal{Y}(\mathcal H/\check {\mathbb P}^N) = \frac{(-1)^{n+1} c_1(\calo(1))^{n}}{(n+2)!} = \frac{(-1)^{n+1}}{(n+2)!} .$$
The consequence
\begin{displaymath}
    \deg \mathcal{Y}(\xc/S)=\frac{(-1)^{n+1}}{(n+2)!}\int_{\xc_0}c_{n+1}^{\xc_0}(\Omega_{\xc/S})\cap [\xc]
\end{displaymath}
follows by the properties of localized Chern classes and since $c_{n+1}^{\xc_0}(\Omega_{\xc/S})$ is supported on the singular locus $\zc$.

\end{proof}


\subsubsection{The Yoshikawa class for Kulikov families of surfaces}
We now look at a germ of a Kulikov family over a disk, $f\colon\xc\to S$. We assume that $\xc$ is regular, $f$ has relative dimension 2 and a unique singular fiber over 0, and finally that the relative canonical sheaf $K_{\xc}$ is trivial. Observe we don't require the generic fiber to be a K3 surface, hence we also allow it to be an abelian surface.

\begin{theo}\label{prop:yoshikawa-kulikov}
The Yoshikawa class of a Kulikov family as above satisfies $$\deg \mathcal{Y}(\xc/S) = \frac{-1}{24} [\chi(\xc_\infty)-\chi(\xc_0)].$$
\end{theo}
\begin{proof} Let $b\colon\xc'\to\xc$ be the Nash blow-up, with universal quotient bundle $\qc$ and exceptional divisor $E$. A direct computation using Lemma \ref{lemma:simplification-bivariant} shows that the degree is given by
\begin{equation}\label{eq:thm-kulikov}
    \deg \mathcal{Y}(\xc/S)=\int_{E}\frac{- c_1( \qc) c_1(\calo(E))-c_1(\qc)^2 - c_2( \qc)}{24}.
\end{equation}
Recall the exact sequence 
\begin{displaymath}
    0\to L_{E}\to b^{\ast}\Omega_{\xc/S}\to\qc\to 0,
\end{displaymath}
that together with the first item in Lemma \ref{lemma:simplification-bivariant} implies
\begin{displaymath}
    c_{1}(b^{\ast}K_{\xc/S}\mid_{E})\cap [\xc']=c_{1}(\calo(E)\mid_{E})\cap [\xc']+c_{1}(\qc\mid_{E})\cap [\xc'].
\end{displaymath}
But by the Kulikov assumption, $K_{\xc/S}$ is trivial, and therefore
\begin{displaymath}
    c_{1}(\calo(E)\mid_{E})\cap [\xc']=-c_{1}(\qc\mid_{E})\cap [\xc'].
\end{displaymath}
Plugging this relation into \eqref{eq:thm-kulikov}, we find
\begin{displaymath}
    \deg \mathcal{Y}(\xc/S)=-\int_{E}\frac{c_{2}(\qc)}{24}.
\end{displaymath}
We conclude by Lemma \ref{prop:bloch}.
\end{proof}
Notice that if $\xc_\infty$ is $K3$ (resp. abelian surface) the Euler characteristic is 24 (resp. 0).


\section{Degeneration of the BCOV metric}\label{section:BCOV}
In this section we will consider families of Calabi--Yau varieties and their BCOV line bundles. More precisely, we will study the BCOV metric introduced by \cite{fang-lu-yoshikawa} and its asymptotic behavior under degeneration. We will use the results in the preceding sections to show that the singularity is governed by topological invariants, especially vanishing cycles in the case of Kulikov families. 

For the rest of this section, let $f\colon\xc\to S$ be a generically smooth flat projective morphism of complex algebraic manifolds with connected fibers, and $\dim S=1$. We suppose the non-singular fibers are $n$-dimensional Calabi--Yau varieties, in the sense that their canonical bundles are trivial. We suppose that $X$ has a fixed K\"ahler metric $h_{X}$. 

\subsection{The BCOV line bundle and metric} We define the BCOV line bundle. First assume that $f$ is smooth. Then we put
\begin{eqnarray*}\lambda_\bcov(\Omega^\bullet_{\xc/S})&:=&\lambda\left(\bigoplus_{0\leq p\leq n}(-1)^pp\Omega^p_{\xc/S}\right)
=\bigotimes_{0\leq p\leq n}\lambda(\Omega^p_{\xc/S})^{(-1)^pp}\\
&=&\bigotimes_{0\leq p,q\leq n}\left(\det\rc ^q f_\ast(\Omega^p_{\xc/S})\right)^{(-1)^{p+q}p}.
\end{eqnarray*}
In general, the sheaves $\Omega^p_{X/S}$ are only coherent sheaves on $\xc$, and not locally free. To extend the BCOV line bundle from the smooth locus to the whole base $S$, it is useful to introduce the so-called K\"ahler resolution of $\Omega_{\xc/S}^{p}$, involving the locally free sheaves $\Omega_{\xc}^{p}$ and $f^{\ast}\Omega_{S}^{\otimes q}$. Equivalently, we apply the left derived functor $L\Lambda^p$ to $\Omega_{\xc/S}$. This is achieved by simply applying the exterior power functors
to the exact sequence~\eqref{sigmastar} defining the relative cotangent sheaf.
For each $0\leq p\leq n$, we obtain  a complex
\begin{eqnarray*}
\widetilde{\Omega^p_{\xc/S}} :
\left(f^\ast \Omega_S\right)^{\otimes p}\to
\left(f^\ast \Omega_S\right)^{\otimes p-1}\otimes \Omega_\xc\to\cdots
\cdots\to
\left(f^\ast \Omega_S\right)\otimes \Omega_\xc^{p-1}\to
\Omega_\xc^p.
\end{eqnarray*}

The \emph{K\"ahler extension} $\lambda_\bcov(\widetilde{\Omega^\bullet_{\xc/S}})$
of the BCOV line bundle on the smooth locus
is then defined to be
$$\lambda_\bcov(\widetilde{\Omega^\bullet_{\xc/S}})
=\lambda\left(\bigoplus_{0\leq p\leq n}(-1)^pp\widetilde{\Omega^p_{\xc/S}}\right)
=\bigotimes_{p=0}^{n}\bigotimes_{j=0}^p
\left(f^\ast \Omega_S\right)^{ (-1)^{p+j}pj}\lambda(\Omega_\xc^{p-j})^{(-1)^{p+j}p}.
$$

For smooth $f$, and depending on the K\"ahler metric $h_{X}$, the BCOV line bundle carries a combination of Quillen metrics. We now introduce the BCOV metric, following \cite[Def. 4.1]{fang-lu-yoshikawa}, but phrased differently.
\begin{definition}
\begin{enumerate}
    \item  The function $A(X/S)\in\mathcal{C}^{\infty}(S)$ is locally given by the formula
    \begin{eqnarray*}
 A(\xc/S)&=& \n\eta_{\xc/S}\n^{\chi(X_{\infty})/6}_{L^2}
 \exp\left\lbrace\frac{(-1)^{n+1}}{12} f_\ast\left(\log(\frac{\n \eta_\xc\n^2}{\n df\n^2})c_n(\Omega_{\xc/S},h_{X})\right)\right\rbrace.
 \end{eqnarray*}
 Here, $\eta_{X}$ is a nowhere vanishing global section of $K_{X}$ (which exists locally relative to the base) and $\eta_{X/S}$ is the section of $f_{\ast}(K_{X/S})$ determined by $\eta_{X}=\eta_{X/S}\wedge f^{\ast}(ds)$, for some local coordinate $s$ on $S$.  
 \item The BCOV metric on $\lambda_{\bcov}(\Omega_{X/S}^{\bullet})$ is
 \begin{displaymath}
    h_{\bcov}=A(X/S)h_{\quillen},
 \end{displaymath}
 where $h_{\quillen}$ is the Quillen metric depending on $h_{X}$.
\end{enumerate}

\end{definition}
The section $\eta_{X/S}$ is sometimes called the Gelfand-Leray form of $\eta$, with respect to $f$. 

The following statement describes the singular behaviour of the BCOV metric when the morphism $f\colon X\to S$ is only supposed generically smooth.

\begin{prop}\label{prop:bcov}
Let $f~:~\xc\to S$ be a generically smooth family of Calabi--Yau varieties of dimension~$n$. Assume there is at most one singular fiber of equation $s=0$. We denote by $\alpha$ and $\beta$ the coefficients encoding the asymptotics of the $L^2$-metric in Proposition \ref{alphabeta}.
\begin{enumerate}
\item
Choose a local holomorphic frame $\widetilde\sigma$ for the K\"ahler extension $\lambda_\bcov(\widetilde{\Omega^\bullet_{\xc/S}})$.
Then the asymptotic of the BCOV norm of $\widetilde\sigma$ is
\begin{eqnarray*}
-\log\n\widetilde{\sigma}\n_\bcov^2
&=&\alpha_{\bcov}\log |s|^2-\frac{\chi(X_\infty)}{12}\beta\log|\log|s|^2| + \hbox{continuous} \\ &=&-\left[\sum_{p=0}^n \sum_{j=0}^p\int_E\left( p(-1)^{j}\Todd^\ast Q
\frac{\Todd^\ast \oc_\XT(E)-e^{(p-j)c_1(\oc_\XT(E))}}{c_1(\oc_\XT(E))}
q^\ast\Ch(\Omega_\xc^j)\right)\right.\\
&& \left.-\frac{1}{12}\left(\alpha \chi(X_\infty) + (\chi(X_\infty)-\chi(X_0))
+(-1)^{n+1}\int_{b^\ast B} c_n(Q)\right)\right]\log |s|^2
-\frac{\chi(X_\infty)}{12}\beta\log|\log|s|^2| + \hbox{continuous},
\end{eqnarray*}
where $X'\stackrel{b}\to\xc$ is the Nash blowup of $f$,
$E$ its exceptional divisor, $Q$ the tautological quotient vector bundle on $X'$ and $B$ is the divisor of the evaluation map \eqref{eq:eval}.
\item  The BCOV metric uniquely extends to a good metric (in the sense of Mumford) on the $\Q$-line bundle
$\lambda_\bcov(\widetilde{\Omega^\bullet_{\xc/S}})\otimes\oc(-\alpha_\bcov [0])$. It has an $L^p$ ($p>1$) potential $-\frac{\chi(X_\infty)}{12}\beta\log|\log|s|^2| + \hbox{continuous}$.

\item Suppose $f\colon X\to S$ is smooth, and is the restriction of a Kuranishi family under a classifying map $\iota$. Then the curvature form of the BCOV metric agrees with the pull-back of the Weil-Petersson form
$$c_{1}(\lambda_\bcov(\Omega^\bullet_{\xc/S}),h_{\bcov})= \frac{\chi(X_\infty)}{12}\iota^\ast \omega_{WP}.$$
\end{enumerate}
\end{prop}
\begin{proof}
The first equality is the conjunction of the asymptotic formulas of the Quillen metric and computations and asymptotics of the term $A(X/S)$. The Quillen part is covered by Theorem \ref{theorem:yoshikawa-quillen}.  For $A(X/S)$, we compute:
\begin{eqnarray*}
\log \ac(\xc/S)&=&\frac{\chi(X_\infty)}{12}\log\n \eta_{\xc/S}\n^2_{L^2}
+ \frac{(-1)^{n+1}}{12}
(f\circ b)_\ast\left(b^\ast\log\left(\n \eta_\xc\n^2\right) c_n(Q)\right)\\
&&+\frac{(-1)^{n}}{12}
(f\circ b)_\ast\left(b^\ast\log\left(\n df\n^2\right) c_n(Q)\right).
\end{eqnarray*}
The asymptotics of the first term are given by the asymptotics of the $L^2$-metric, established in Proposition \ref{alphabeta}
\begin{eqnarray*}
 -\log\n \eta_{\xc/S}\n^2_{L^2}
= \alpha\log |s|^2 - \beta (\log|\log|s|^2|) + \hbox{continuous}.
\end{eqnarray*}
The second term and the third terms have asymptotics given by~\cite[Lemma 4.4 and Corollary 4.6]{yoshikawa}
$$(f\circ b)_\ast\left(b^\ast\log\left(\n \eta_\xc\n^2\right) c_n(Q)\right)
=\left(\int_{b^\ast B} c_n(Q)\right)\log |s|^2+\hbox{continuous}$$
$$(f\circ b)_\ast\left(b^\ast\log\left(\n df\n^2\right) c_n(Q)\right)
=\left(\int_{E} c_n(Q)\right)\log |s|^2+\hbox{continuous}.$$
For the first equality, we have used $div(\eta_{\xc})= B$ and for the second equality we have used that the zero-locus of $df$ is exactly the singular locus $Z$ and $E = b^{-1}(Z)$. We obtain the final form by applying the formula $\int_{E} c_n(Q) = (-1)^n (\chi(X_\infty) - \chi(X_0))$

The second part of the proposition is a consequence of the first. Indeed, by Theorem \ref{theorem:yoshikawa-quillen} it is enough to provide Mumford-good estimates on the continuous rests of the formulas above. But they are also as in Remark \ref{rem:good}, by the same \cite[Lemma 4.4]{yoshikawa}, and hence good in the sense of Mumford. 

The third part is \cite[Thm.4.9]{fang-lu-yoshikawa}.
\end{proof}

\begin{remark}
If the singular fibers of $f\colon\xc\to S$ have normal crossings, one may as well consider the logarithmic extension of $\lambda_{\bcov}$, by using the complex $\Omega_{\xc/S}^{\bullet}(\log)$ instead of the K\"ahler resolution. The logarithmic and the K\"ahler extensions can explicitly be compared by means of the residue exact sequence. 
\end{remark}

\subsection{Computation of $\alpha_\bcov$}\label{alphabcov}

The asymptotic formulas provided by \cite[Thm. 5.4]{fang-lu-yoshikawa} and Proposition \ref{prop:bcov} above are cumbersome, and the relation to topological invariants (for instance vanishing cycles) is not clear. We next show that several simplifications and cancellations occur in the expression defining $\alpha_{\bcov}$. We rewrite it solely in terms of the characteristic classes $c_{n}(\qc)$, $c_{1}(\qc)c_{n-1}(\qc)$ and $c_{1}(b^{\ast}K_{\xc})c_{n-1}(\qc)$. We derive consequences for Kulikov type families. 

Recall that $b:\xc'\to\xc$ denotes the Nash blowup of the morphism $f:\xc\to S$, with exceptional divisor $E$ and universal quotient bundle $\qc$. We focus on the combination of characteristic classes
\begin{displaymath}
    \omega:=\Td^{\ast} (\qc\mid_{E})\sum_{p=0}^n \sum_{j=0}^p p(-1)^{j}
\left(\Td^{\ast}(\oc(E)\mid_{E})-\Ch(\oc(E)\mid_{E})^{p-j}\right)
\Ch(b^{\ast}\Omega_\xc^j \mid_{E})\cap [\xc'].
\end{displaymath}
To simplify the discussion, we supress the $\cap [\xc']$ from the notations. In the definition of $\alpha_{\bcov}$, the class $\omega$ contributes through
\begin{displaymath}
   - \int_{E}\frac{\omega}{c_{1}(\oc(E)\mid_{E})}.
\end{displaymath}
Because of the division by $c_{1}(\oc(E)\mid_{E})$ and since $E$ is a divisor in $\xc'$, we only seek a simple expression for the degree $n+1$ part of $\omega$. \emph{A priori}, we know this component has to be a multiple $c_{1}(\oc(E)\mid_{E})$.

The point of departure is to restrict the universal exact sequence
\begin{displaymath}
    0 \to L_E \to b^* \Omega_{\xc/S} \to \qc \to 0
\end{displaymath}
to the exceptional divisor. Because $\qc$ is locally free, the restriction of the sequence to $E$ remains exact. Moreover, we observe that $E$ lies above the singular locus $\zc$ of the morphism $f:\xc\to S$, and hence $b^{\ast}\Omega_{\xc/S}\mid_{E}=b^{\ast}\Omega_{\xc}\mid_{E}$. Therefore, we obtain an exact sequence
\begin{equation}\label{eq:univ-quot-restr}
    0 \to L_E \to b^* \Omega_{\xc}\mid_{E} \to  \qc\mid_{E} \to 0.
\end{equation}
We also recall from Lemma \ref{lemma:simplification-bivariant} that $L_{E}$ is a line bundle on $E$, and that as a bivariant class with values in $A_{\ast}(E)$, the relation $c_{1}(L_E)=c_{1}(\oc(E)_{\mid E})$ holds. Taking exterior powers in \eqref{eq:univ-quot-restr} and substituting $c_{1}(L_E)$ by $c_{1}(\oc(E)_{\mid E})$, we find 
\begin{equation}\label{eq:ch-omega-j}
    \Ch(b^{\ast}\Omega_{\xc}^{j} \mid_{E})=\Ch(\Lambda^{j}\qc\mid_{E})+
    \Ch(\Lambda^{j-1}\qc\mid_{E})\Ch(\oc(E)\mid_{E}),
\end{equation}
with the convention that $\Lambda^{j-1}\qc=0$ for $j=0$. From now on, to lighten notations, we also skip the restriction to $E$ from the notations, by saying instead that a given relation holds on E. Therefore, on $E$ we can write $\omega=\vartheta+\vartheta'$, where
\begin{displaymath}
    \vartheta=(\Td^{\ast}(\qc))(\Td^{\ast}\oc(E))\sum_{0\leq j\leq p\leq n}p(-1)^{j}(\Ch(\Lambda^{j}\qc)+\Ch(\Lambda^{j-1}\qc)\Ch(\oc(E)))
\end{displaymath}
and the class $\vartheta'$ is defined to be the rest. Actually, after a simple telescopic sum, $\vartheta'$ simplifies to
\begin{displaymath}
    \vartheta'=-\Td^{\ast}(\qc)\sum_{p=0}^{n}p(-1)^{p}\Ch(\Lambda^{p}\qc).
\end{displaymath}
We now work on the class $\omega$.
\begin{lem}\label{lemma:omega-bcov}
The class $\vartheta$ is the sum of four contributions
\begin{align*}
    &\vartheta_{1}=-\frac{n(n+1)}{2}(-1)^{n}c_{n}(\qc)c_{1}(\oc(E))+hct,\\
    &\vartheta_{2}=-\Td^{\ast}(\qc)\Td^{\ast}(\oc(E))\sum_{j=0}^{n}(-1)^{j+1}\frac{j(j+1)}{2}\Ch(\Lambda^{j}\qc)\Ch(\oc(E)) \\
    &\vartheta_{3}=-\Td^{\ast}(\qc)\Td^{\ast}(\oc(E))\sum_{j=0}^{n}(-1)^{j}\frac{j(j-1)}{2}\Ch(\Lambda^{j}\qc),
\end{align*}
where $hct$ is a shortcut for "higher codimension terms".
\end{lem}
\begin{proof}
The proof is elementary, and relies on the property \cite[Example 3.2.5]{Fulton}
\begin{equation}\label{eq:chern-koszul}
    \Td^{\ast}(\qc)\sum_{p=0}^{n}\Ch(\Lambda^{p}\qc)=(-1)^{n}c_{n}(\qc),
\end{equation}
and the power series expansions of $\Td^{\ast}(\oc(E))$ and $\Ch(\oc(E))$ in $c_{1}(\oc(E))$.
\end{proof}
The relation \eqref{eq:chern-koszul} and the expressions for the classes $\vartheta$ and $\vartheta'$ motivate the following definition.
\begin{definition}
For a vector bundle $F$ of rank $r$, we define
\begin{align*}
    &P(F)=\Td^{\ast}(F)\sum_{p=0}^{r}(-1)^{p}\Ch(\Lambda^{p}F)\quad(=(-1)^{r}c_{r}(F)),\\
    &P'(F)=\Td^{\ast}(F)\sum_{p=0}^{r}(-1)^{p}p\Ch(\Lambda^{p}F),\\
    &P''(F)=\Td^{\ast}(F)\sum_{p=0}^{r}(-1)^{p}\frac{p(p-1)}{2}\Ch(\Lambda^{p}F).
\end{align*}
\end{definition}
As the notation suggests, the classes $P'(F)$ and $P''(F)$ are to be seen as the first and second derivatives of $P(F)$. More precisely, we have
\begin{lem}\label{lemma:derivatives}
The classes $P$, $P'$ and $P''$ satisfy
\begin{align*}
    &P(F\oplus G)=P(F)P(G),\\
    &P'(F\oplus G)=P'(F)P(G)+P(F)P'(G),\\
    &P''(F\oplus G)=P''(F)P(G)+P'(F)P'(G)+P(F)P''(G).
\end{align*}
In particular, given line bundles $L_{1},\ldots,L_{r}$, we have
\begin{align*}
    &P'(L_{1}\oplus\ldots\oplus L_{r})=\sum_{i=1}^{r}P(L_{1})\ldots P'(L_{i})\ldots P(L_{r}),\\
    &P''(L_{1}\oplus\ldots\oplus L_{r})=\sum_{1\leq i<j\leq r}P(L_{1})\ldots P'(L_{i})\ldots P'(L_{j})\ldots P(L_{r}).
\end{align*}
\end{lem}
\begin{proof}
The first part is an easy computation using the multiplicativity of $\Td^{\ast}$ with respect to direct sums of vector bundles, and the multiplicativity of $\Ch$ with respect to tensor products of vector bundles. The conclusion for direct sums of line bundles requires the observation $P''(L)=0$ for a line bundle $L$.
\end{proof}
In terms of $P'$ and $P''$, the classes $\vartheta'$ and $\vartheta_{3}$ are
\begin{align*}
    &\vartheta'=-P'(\qc),\\
    &\vartheta_{3}=-\Td^{\ast}(\oc(E))P''(\qc).
\end{align*}
For $\vartheta_{2}$, an easy computation gives the string of equalities,
\begin{equation}\label{eq:theta-4}
    \begin{split}
     \vartheta_{2}&=(P'(\qc)+P''(\qc))\Ch(\oc(E))\Td^{\ast}(\oc(E))\\
     &=(P'(\qc)+P''(\qc))c_{1}(\oc(E))+P'(\qc)\Td^{\ast}(\oc(E))-\vartheta_{3}+hct\\
     &=P'(\qc)\left(\frac{1}{2}c_{1}(\oc(E))+\frac{1}{12}c_{1}(\oc(E))^{2}\right)
     +P''(\qc)c_{1}(\oc(E))-\vartheta_{3}-\vartheta'+hct.
     \end{split}
\end{equation}
To conclude, we thus have to extract the codimension $n-1$ and $n$ parts of the class $P'(\qc)$, and the codimension $n$ class of $P''(\qc)$.
\begin{lem}\label{lemma-simpli-derivatives}
\begin{enumerate}
\item For the first derivative class, we have
\begin{displaymath}
    P'(\qc)=(-1)^{n}c_{n-1}(\qc)+(-1)^{n}\frac{n}{2}c_{n}(\qc)+hct,
\end{displaymath}
\item For the second derivative class, we have
\begin{displaymath}
    P''(\qc)^{(n)}=(-1)^{n}\frac{n(3n-5)}{24}c_{n}(\qc)+(-1)^{n}\frac{1}{12}c_{1}(\qc)c_{n-1}(\qc).
\end{displaymath}
\end{enumerate}
\end{lem}
\begin{proof}
By Lemma \ref{lemma:derivatives} and the splitting principle, we can suppose that $\qc$ splits into a direct sum of line bundles $L_{1},\ldots,L_{n}$. 

For the first item, we use the formula for $P'(L_{1}\oplus\ldots\oplus L_{n})$ in Lemma \ref{lemma:derivatives}. 
For this, we recall from \eqref{eq:chern-koszul}
\begin{displaymath}
    P(L_{i})=-c_{1}(L_{i})
\end{displaymath}
and observe
\begin{displaymath}
    \begin{split}
    P'(L_{i})&=P(L_{i})-\Td^{\ast}(L_{i})\\
    &=-1-\frac{1}{2}c_{1}(L_{i})-\frac{1}{12}c_{1}(L_{i})^{2}+hct.
    \end{split}
\end{displaymath}
After an elementary computation, one concludes by taking into account
\begin{align*}
    &c_{n}(\qc)=c_{1}(L_{1})\ldots c_{1}(L_{n}),\\
    &c_{n-1}(\qc)=\sum_{i=1}^{n}c_{1}(L_{1})\ldots \widehat{c_{1}(L_{i})}\ldots c_{1}(L_{n}).
\end{align*}
For the second item, we proceed similarly. We first compute
\begin{displaymath}
    P'(L_{i})P'(L_{j})=\frac{1}{4}c_{1}(L_{i})c_{1}(L_{j})
    +\frac{1}{12}c_{1}(L_{i})^{2}+\frac{1}{12}c_{1}(L_{j})^{2}+hct.
\end{displaymath}
Hence, we obtain
\begin{displaymath}
    \begin{split}P''(\qc)^{(n)}=&(-1)^{n}\frac{n(n-1)}{8}c_{n}(\qc)\\
    &+(-1)^{n}\frac{1}{12}\sum_{i<j}c_{1}(L_{1})\ldots c_{1}(L_{i})^{2}\ldots\widehat{c_{1}(L_{j})}\ldots c_{n}(L_{n})\\
    &+(-1)^{n}\frac{1}{12} \sum_{i<j}c_{1}(L_{1})\ldots\widehat{c_{1}(L_{i})}\ldots c_{1}(L_{j})^{2}\ldots c_{1}(L_{n}).
    \end{split}
\end{displaymath}
But we observe
\begin{displaymath}
    \begin{split}
        \sum_{i<j}&c_{1}(L_{1})\ldots c_{1}(L_{i})^{2}\ldots\widehat{c_{1}(L_{j})}\ldots c_{n}(L_{n})
    +\sum_{i<j}c_{1}(L_{1})\ldots\widehat{c_{1}(L_{i})}\ldots c_{1}(L_{j})^{2}\ldots c_{1}(L_{n})\\
    &=(c_{1}(L_{1})+\ldots+c_{1}(L_{n}))\sum_{i=1}^{n}c_{1}(L_{1})\ldots\widehat{c_{1}(L_{i})}\ldots c_{1}(L_{n})
    -nc_{1}(L_{1})\ldots c_{1}(L_{n})\\
    &=c_{1}(\qc)c_{n-1}(\qc)-nc_{n}(\qc).
    \end{split}
\end{displaymath}
All in all, we conclude
\begin{displaymath}
    P''(\qc)^{(n)}=(-1)^{n}\frac{n(3n-5)}{24}c_{n}(\qc)+(-1)^{n}\frac{1}{12}c_{1}(\qc)c_{n-1}(\qc),
\end{displaymath}
as was to be shown.
\end{proof}
\begin{prop}\label{prop:omega-c1}
The class $\omega$ satisfies
\begin{displaymath}
    \int_{E}\frac{\omega}{c_{1}(\oc(E)\mid_{E})}=
    (-1)^{n+1}\frac{9n^{2}+11n}{24}\int_{E}c_{n}(\qc)
    +\frac{(-1)^{n}}{12}\int_{E}b^{\ast}c_{1}(K_{\xc})c_{n-1}(\qc).
\end{displaymath}
\end{prop}
\begin{proof}
We collect the identities in Lemma \ref{lemma:omega-bcov}, the expression \eqref{eq:theta-4} for $\vartheta_{4}$ and the values provided by Lemma \ref{eq:univ-quot-restr}. We then observe that
\begin{displaymath}
    c_{1}(\oc(E)\mid_{E})+c_{1}(\qc\mid_{E})=c_{1}(b^{\ast}K_{\xc}\mid_{E}),
\end{displaymath}
as follow from \eqref{eq:univ-quot-restr} and $c_{1}(L_{E})=c_{1}(\oc(E)\mid_{E})$. This concludes the proof.
\end{proof}
\begin{cor}
Suppose that $K_{X}$ is trivial on the singular locus $Z$. Then 
\begin{displaymath}
    \int_{E}\frac{\omega}{c_{1}(\oc(E)\mid_{E})}=
    -\frac{9n^{2}+11n}{24}\left(\chi(\xc_{\infty})-\chi(\xc_{0})\right).
\end{displaymath}
In particular, if $f$ has isolated singularities, then
\begin{displaymath}
     \int_{E}\frac{\omega}{c_{1}(\oc(E)\mid_{E})}=(-1)^{n+1}\frac{9n^{2}+11n}{24}\sum_{x\in X_{0}}\mu_{X,x}.
\end{displaymath}
\end{cor}
\begin{proof}
By applying the projection formula, one infers $$\int_{E}b^{\ast}c_{1}(K_{\xc})c_{n-1}(\qc) = \int_{Z} c_{1}(K_{\xc})b_\ast c_{n-1}(\qc).$$ By assumption, $K_{X}$ is trivial on $Z$, and hence this intersection number vanishes. We conclude by applying the formula  
\begin{displaymath}
    (-1)^{n}\int_{E}c_{n}(\qc)=\chi(\xc_{\infty})-\chi(\xc_{0}).
\end{displaymath}
\end{proof}

\begin{cor}\label{cor:alpha-bcov-kulikov}
The coefficient $\alpha_{\bcov}$ is given by
\begin{displaymath}
    \alpha_{\bcov}=\frac{9n^{2}+11n+2}{24}(\chi(\xc_{\infty})-\chi(\xc_{0}))
    +\frac{\alpha}{12}\chi(X_{\infty}) + \frac{(-1)^n}{12} \int_{B} c_n(\Omega_{X/S}).
\end{displaymath}
\end{cor}
\begin{proof}
For the first assertion, notice that $$b^{\ast}c_{1}(K_{\xc})c_{n-1}(\qc)\cap [E] = c_{n-1}(Q) \cap c_1(b^{\ast} B) \cap [E] = c_{n-1}(Q) \cap c_1(E) \cap [b^\ast B]$$ in the Chow group of the special fiber of $X'\to S$. This is a consequence of the commutativity of intersection classes of Cartier divisors \cite[Sec. 2.4]{Fulton}, and the definition of $c_{1}$ of a line bundle.

Moreover, from Lemma \ref{lemma:L_E} we have $L_{E}\simeq \calo(E)_{\mid E}$ and applying Chern classes on the tautological exact sequence on the Nash blowup, we easily deduce from the Whitney formula that $$c_n(b^{\ast}\Omega_{X/S}) \cap [b^\ast B] = c_n(Q) \cap [b^\ast B] + c_{n-1}(Q) c_1(E) \cap [b^\ast B].$$ Observe that $c_{n}(b^{\ast}\Omega_{X/S})=b^{\ast}c_{n}(\Omega_{X/S})$, because $\Omega_{X/S}$ admits a tow term locally free resolution and $b$ is birational. Applying the projection formula, we finally find
\begin{displaymath}
    \int_{E}b^{\ast}c_{1}(K_{\xc})c_{n-1}(\qc)=\int_{B} c_n(\Omega_{X/S})
    -\int_{b^\ast B}c_{n}(\qc).
\end{displaymath}
We finish the proof of the first claim by plugging this relation into Proposition \ref{prop:omega-c1}, and by the very definition of $\alpha_{\bcov}$.

\end{proof}
To sum up, we conclude by restating Proposition \ref{prop:bcov} (a) for Kulikov families.

\begin{prop}\label{prop:final-bcov}
Let $f~:~\xc\to S$ be a generically smooth family of Calabi--Yau varieties of dimension~$n$, with a unique singular fiber of equation $s=0$. Assume that $X$ is a Kulikov family, i.e. that $B = \emptyset$ (\emph{e.g.} if $K_{X}$ is trivial). 
Choose a local holomorphic frame $\widetilde\sigma$ for the K\"ahler extension $\lambda_\bcov(\widetilde{\Omega^\bullet_{\xc/S}})$. Then the asymptotic of the BCOV norm of $\widetilde\sigma$ is
\begin{eqnarray*}
-\log\n\widetilde{\sigma}\n_\bcov^2
&=&\alpha_{\bcov}\log |s|^2-\frac{\chi(X_\infty)}{12}\beta\log|\log|s|^2| + \hbox{continuous} \\  &=&\left[\frac{9n^{2}+11n+2}{24}(\chi(\xc_{\infty})-\chi(\xc_{0}))
    +\frac{\alpha}{12}\chi(X_{\infty})\right]\log|s|^{2}\\
& & -\frac{\chi(X_\infty)}{12}\beta\log|\log|s|^2| + \hbox{continuous},
\end{eqnarray*}
where $\alpha$ and $\beta$ are as in Proposition \ref{alphabeta}. 
\end{prop}
\begin{cor}
If $n\geq 2$ and $f\colon X\to S$ has only isolated ordinary quadratic singularities, then
\begin{displaymath}
    -\log\n\widetilde{\sigma}\n_\bcov^2=\frac{9n^{2}+11n+2}{24}\#\mathrm{sing}(X_{0})\log |s|^{2} + \hbox{continuous}.
\end{displaymath}
\end{cor}
\begin{proof}
We observed in section \ref{generalities} that a Calabi-Yau degeneration with isolated singularities is automatically Kulikov. The claim then follows from Proposition \ref{prop:final-bcov} together with Proposition \ref{quadraticordinary} and Remark \ref{rem:quadraticsingularity}.
\end{proof}

\bibliographystyle{alpha}
\bibliography{canonical21}{}

\begin{thebibliography}{BGS88b}

\bibitem[Abb00]{Abbes}
A.~Abbes.
\newblock Cycles on arithmetic surfaces.
\newblock {\em Compositio Math.}, 122(1):23--111, 2000.

\bibitem[BB90]{bismutbost}
J.-M. Bismut and J.-B. Bost.
\newblock Fibr\'es d\'eterminants, m\'etriques de {Q}uillen et
  d\'eg\'en\'erescence des courbes.
\newblock {\em Acta Math.}, 165(1-2):1--103, 1990.

\bibitem[Ber16]{Berman}
Robert~J. Berman.
\newblock K-polystability of {${\Bbb Q}$}-{F}ano varieties admitting
  {K}\"ahler-{E}instein metrics.
\newblock {\em Invent. Math.}, 203(3):973--1025, 2016.

\bibitem[BGS88a]{BGS1}
J.-M. Bismut, H.~Gillet, and C.~Soul{\'e}.
\newblock Analytic torsion and holomorphic determinant bundles. {I}.
  {B}ott-{C}hern forms and analytic torsion.
\newblock {\em Comm. Math. Phys.}, 115(1):49--78, 1988.

\bibitem[BGS88b]{BGS2}
J.-M. Bismut, H.~Gillet, and C.~Soul{\'e}.
\newblock Analytic torsion and holomorphic determinant bundles. {II}. {D}irect
  images and {B}ott-{C}hern forms.
\newblock {\em Comm. Math. Phys.}, 115(1):79--126, 1988.

\bibitem[BGS88c]{BGS3}
J.-M. Bismut, H.~Gillet, and C.~Soul{\'e}.
\newblock Analytic torsion and holomorphic determinant bundles. {III}.
  {Q}uillen metrics on holomorphic determinants.
\newblock {\em Comm. Math. Phys.}, 115(2):301--351, 1988.

\bibitem[Bis97]{bismut}
Jean-Michel Bismut.
\newblock Quillen metrics and singular fibres in arbitrary relative dimension.
\newblock {\em J. Algebraic Geom.}, 6(1):19--149, 1997.

\bibitem[BJ]{BoucksomJonsson}
S.~Boucksom and M.~Jonsson.
\newblock Tropical and non-archimedean limits of degenerating families of
  volume forms.
\newblock {\em arXiv:1605.05277 [math.DG]}.

\bibitem[CK82]{Cattani-kaplan}
E.~Cattani and A.~Kaplan.
\newblock Polarized mixed {H}odge structures and the local monodromy of a
  variation of {H}odge structure.
\newblock {\em Invent. Math.}, 67(1):101--115, 1982.

\bibitem[Eri12]{ErikssonQuillen}
D.~Eriksson.
\newblock Degenerating {R}iemann surfaces and the {Q}uillen metric.
\newblock {\em International Mathematics Research Notices}, 2012.

\bibitem[Eri16]{dennis}
D.~Eriksson.
\newblock Discriminants and {A}rtin conductors.
\newblock {\em J. Reine Angew. Math.}, 712:107--121, 2016.

\bibitem[FLY08]{fang-lu-yoshikawa}
H.~Fang, Z.~Lu, and K.-I. Yoshikawa.
\newblock Analytic torsion for {C}alabi-{Y}au threefolds.
\newblock {\em J. Differential Geom.}, 80(2):175--259, 2008.

\bibitem[FM00]{fujino-mori}
O.~Fujino and S.~Mori.
\newblock A canonical bundle formula.
\newblock {\em J. Differential Geom.}, 56(1):167--188, 2000.

\bibitem[Ful98]{Fulton}
W.~Fulton.
\newblock {\em Intersection theory}, volume~2 of {\em Ergebnisse der Mathematik
  und ihrer Grenzgebiete. 3. Folge. A Series of Modern Surveys in Mathematics
  [Results in Mathematics and Related Areas. 3rd Series. A Series of Modern
  Surveys in Mathematics]}.
\newblock Springer-Verlag, Berlin, second edition, 1998.

\bibitem[GKZ08]{GKZ}
I.~M. Gelfand, M.~M. Kapranov, and A.~V. Zelevinsky.
\newblock {\em Discriminants, resultants and multidimensional determinants}.
\newblock Modern Birkh\"auser Classics. Birkh\"auser Boston, Inc., Boston, MA,
  2008.

\bibitem[HN12]{NicaiseHalle}
L.~H. Halle and J.~Nicaise.
\newblock Motivic zeta functions for degenerations of abelian varieties and
  {C}alabi-{Y}au varieties.
\newblock In {\em Zeta functions in algebra and geometry}, volume 566 of {\em
  Contemp. Math.}, pages 233--259. Amer. Math. Soc., Providence, RI, 2012.

\bibitem[Kaw82]{kawa82}
Y.~Kawamata.
\newblock Kodaira dimension of algebraic fiber spaces over curves.
\newblock {\em Invent. Math.}, 66(1):57--71, 1982.

\bibitem[KM76]{KnudsenMumford}
F.~Knudsen and D.~Mumford.
\newblock The projectivity of the moduli space of stable curves. {I}.
  {P}reliminaries on ``det'' and ``{D}iv''.
\newblock {\em Math. Scand.}, 39(1):19--55, 1976.

\bibitem[KN94]{KawamataNamikawa}
Y.~Kawamata and Y.~Namikawa.
\newblock Logarithmic deformations of normal crossing varieties and smoothing
  of degenerate {C}alabi-{Y}au varieties.
\newblock {\em Invent. Math.}, 118(3):395--409, 1994.

\bibitem[Kod64]{kodaira:canonical}
K.~Kodaira.
\newblock On the structure of compact complex analytic surfaces. {I}.
\newblock {\em Amer. J. Math.}, 86:751--798, 1964.

\bibitem[Kol97]{kollar:singpairs}
J.~Koll{\'a}r.
\newblock Singularities of pairs.
\newblock volume~62 of {\em Proc. Sympos. Pure Math.}, pages 221--287. Amer.
  Math. Soc., Providence, RI, 1997.

\bibitem[KS04]{Kato-Saito}
K.~Kato and T.~Saito.
\newblock On the conductor formula of {B}loch.
\newblock {\em Publ. Math. Inst. Hautes \'Etudes Sci.}, (100):5--151, 2004.

\bibitem[Kul77]{kulikov}
V.~S. Kulikov.
\newblock Degenerations of {$K3$} surfaces and {E}nriques surfaces.
\newblock {\em Izv. Akad. Nauk SSSR Ser. Mat.}, 41(5):1008--1042, 1199, 1977.

\bibitem[Lee10]{Lee}
N.-H. Lee.
\newblock Calabi-{Y}au construction by smoothing normal crossing varieties.
\newblock {\em Internat. J. Math.}, 21(6):701--725, 2010.

\bibitem[MT09]{MT}
C.~Mourougane and S.~Takayama.
\newblock Extension of twisted {H}odge metrics for {K}\"ahler morphisms.
\newblock {\em J. Differential Geom.}, 83(1):131--161, 2009.

\bibitem[Mum77]{mumford}
D.~Mumford.
\newblock Hirzebruch's proportionality theorem in the noncompact case.
\newblock {\em Invent. Math.}, 42:239--272, 1977.

\bibitem[Pie79]{piene}
R.~Piene.
\newblock Ideals associated to a desingularization.
\newblock In {\em Algebraic geometry ({P}roc. {S}ummer {M}eeting, {U}niv.
  {C}openhagen, {C}openhagen, 1978)}, volume 732 of {\em Lecture Notes in
  Math.}, pages 503--517. Springer, Berlin, 1979.

\bibitem[Sch73]{schmid}
W.~Schmid.
\newblock Variation of {H}odge structure: the singularities of the period
  mapping.
\newblock {\em Invent. Math.}, 22:211--319, 1973.

\bibitem[Ste77]{Steenbrink-mixedonvanishing}
J.~H.~M. Steenbrink.
\newblock Mixed {H}odge structure on the vanishing cohomology.
\newblock In {\em Real and complex singularities ({P}roc. {N}inth {N}ordic
  {S}ummer {S}chool/{NAVF} {S}ympos. {M}ath., {O}slo, 1976)}, pages 525--563.
  Sijthoff and Noordhoff, Alphen aan den Rijn, 1977.

\bibitem[Yos98]{yoshikawa2}
K.-I. Yoshikawa.
\newblock Smoothing of isolated hypersurface singularities and {Q}uillen
  metrics.
\newblock {\em Asian J. Math.}, 2(2):325--344, 1998.

\bibitem[Yos07]{yoshikawa}
K.-I. Yoshikawa.
\newblock On the singularity of {Q}uillen metrics.
\newblock {\em Math. Ann.}, 337(1):61--89, 2007.

\bibitem[Yos10]{yoshikawa3}
K.-I. Yoshikawa.
\newblock Singularities and analytic torsion.
\newblock {\em arXiv:1007.2835 [math.AG]}, 2010.

\end{thebibliography}

\end{document}